\documentclass{article}
\title{CWENO: uniformly accurate reconstructions for balance laws}
\author{I. Cravero\thanks{Universit\`a di Torino, Italy.}, 
G. Puppo\thanks{Universit\`a dell'Insubria, Como, Italy.}, 
M. Semplice\thanks{Universit\`a di Torino, Italy. Correspondence to matteo.semplice@unito.it}, 
G. Visconti\thanks{Universit\`a dell'Insubria, Como, Italy.}
}

\usepackage{graphicx}

\usepackage{filemod}
\date{Version: \FilemodToday{\jobname.tex}}

\usepackage{fancyhdr}
\usepackage{amsmath}
\DeclareMathOperator{\CWENO}{\mathsf{CWENO}}

\DeclareMathOperator{\WENO}{\mathsf{WENO}}
\DeclareMathOperator{\diam}{diam}
\DeclareMathOperator{\sign}{sign}

\usepackage{amsthm, amsfonts}
\newtheorem{definition}{Definition}
\newtheorem{example}{Example}
\newtheorem{remark}{Remark}
\newtheorem{summary}{Summary}

\newtheorem{proposition}{Proposition}
\newtheorem{theorem}{Theorem}

\newtheorem{test}{Test}

\usepackage{nicefrac}
\newcommand{\um}{\nicefrac{1}{2}}

\newcommand{\Ogrande}{\mathcal{O}}

\newcommand{\dx}{\mathrm{d}x}
\newcommand{\R}{\mathbb{R}}
\newcommand{\Popt}{P_{\text{\sf opt}}}
\newcommand{\Prec}{P_{\text{\sf rec}}}
\newcommand{\Precj}{P_{\text{\sf rec},j}}
\newcommand{\Poly}[1]{\mathbb{P}^{#1}}
\newcommand{\grado}{g}

\newcommand{\ca}[1]{\overline{#1}}

\newcommand{\emme}{\widehat{m}}

\begin{document}
\maketitle

\paragraph{Abstract}
In this paper we introduce a general framework for defining and studying essentially non-oscillatory reconstruction procedures of arbitrarily high order accuracy, interpolating data in a central stencil around a given computational cell ($\CWENO$). This technique relies on the same selection mechanism of  smooth stencils adopted in $\WENO$, but here the pool of candidates for the selection includes polynomials of different degrees. This seemingly minor difference allows to compute an analytic expression of a polynomial interpolant, approximating the unknown function uniformly within a cell, instead of only at one point at a time. 
For this reason this technique is particularly suited for balance laws for finite volume schemes, when averages of source terms require high order quadrature rules based on several points; in the computation of local averages, during refinement in h-adaptive schemes; or in the transfer of the solution between grids in moving mesh techniques, and in general when a globally defined reconstruction is needed. Previously, these needs have been satisfied mostly by ENO reconstruction techniques, which, however, require a much wider stencil then the $\CWENO$ reconstruction studied here, for the same accuracy.

\paragraph{MSC} 65M08, 65M12.

\paragraph{Keywords} high order accuracy, essentially non oscillatory, finite volume schemes, balance laws, non uniform grids.

\section{Introduction}

\paragraph{Motivation.} Conservation laws arise in many fields in applied  mathematics, such as gas dynamics, magneto-hydrodynamics, or even traffic flow. When a source term is present, these equations are called balance laws, and an even wider field of applications opens up. Balance laws describe in fact phenomena in environmental or meteorological fields, plasmas,  astrophysics. 

In many cases, fast and efficient algorithms are crucial, and this means to be able to provide robust high order schemes, which yield accurate solutions even on coarse grids. Moreover, it is important to be able to implement such schemes on adaptive, and therefore non uniform, grids. This paper is concerned with the analysis of a class of algorithms that, starting from a set of data, permit to reconstruct with high order accuracy a representation in space of the underlying function.

We start from a system of balance laws
\begin{equation}\label{eq:TheEq}
\partial_t u + \sum_{i=1}^n \partial_{x_i} f_i (u) = s(u;x,t).
\end{equation}
Here $u(x,t) : \mathbb{R}^n\times \mathbb{R}^+ \rightarrow \mathbb{R}^m$ is the unknown function, $n$ is the number of space dimensions, $m$ is the number of equations, and $t$ denotes time. The functions $f_i(u) : \mathbb{R}^m\rightarrow \mathbb{R}^m$ are called fluxes, and usually they are  smooth known functions of $u$, with Jacobians $\sum_i v_i f'_i$ diagonalizable with real eigenvalues, along all possible directions $v \in \mathbb{R}^n$. Finally, $s: \mathbb{R}^m\times \mathbb{R}^n\times \mathbb{R}^+\to\R^m$ is the source term, which is a known, bounded function of the unknown $u$, but it also  may depend on space (as in the shallow water equations), or even time.
Suppose the equation is defined on a domain $\mathcal{D} \subseteq \mathbb{R}^n$, with suitable initial and boundary conditions. 

To integrate this system of  equations numerically, one must define a {\em grid} in the domain $\mathcal{D}$. In this work, we will propose schemes that are directly applicable  when the grid is a non uniform globally Cartesian grid, so that $\mathcal{D}$ is covered by the union of rectangles $\mathcal{D}\subseteq \bigcup \; \Omega_k$. Note that boundary conditions for general $\mathcal{D}$ could then be dealt with immersed boundary techniques, see e.g. \cite{Iollo:12}.

On each cell $\Omega_k$, define the {\em cell average} of the solution,
\begin{equation}
\ca{u}_k (t) = \frac{1}{\left|\Omega_k\right|} \int_{\Omega_k} u(x,t)\; \dx.
\end{equation}
Using the method of lines, we integrate the balance law \eqref{eq:TheEq} on each of the $\Omega_k$, obtaining the finite volume formulation
\begin{equation}
\frac{\mathrm{d}}{\mathrm{d} t} \ca{u}_k = 
- \frac{1}{\left|\Omega_k\right|} \int_{\partial \Omega_k} {\bf f} \cdot {\bf n}_k 
+ \frac{1}{\left|\Omega_k\right|} \int_{\Omega_k} s(u;x,t) \; \dx, \label{eq:TheSemidis}
\end{equation}
where ${\bf f} = [f_1,\dots,f_n]$ and ${\bf n}_k$ is the outward normal to the cell $\Omega_k$. 
To transform \eqref{eq:TheSemidis} in a Finite Volume numerical scheme, a recipe for the evaluation of the fluxes across the cell boundary must be provided, together with a numerical method to integrate the resulting system of ODE's. This process must involve a {\em reconstruction} algorithm that, starting from the cell averages at a given time $t$, reconstructs approximate values of the solution $u$ in all the quadrature points along the contour $\partial \Omega_k$ of each cell (to evaluate the numerical fluxes) and in all quadrature nodes within $\Omega_k$ (to compute the cell average of the source). The purpose of the present work is to study a class of reconstructions which provide an approximation of the underlying solution which is {\em uniformly} accurate within the whole cell. In this fashion, the reconstruction can be evaluated simultaneously on all quadrature points needed to approximate \eqref{eq:TheSemidis}, thus only one reconstruction step is needed for each evaluation of the right hand side.

\paragraph{Background.}
A very popular algorithm to compute the reconstruction in high order finite volume  schemes for conservation and balance laws is WENO (Weighted Essentially Non Oscillatory), see the seminal paper \cite{JiangShu:96} and the reviews \cite{ Shu97,Shu:2009:WENOreview}, but the literature on this technique is huge. WENO is based on a piecewise reconstructing polynomial that reproduces a high order polynomial using data from a wide stencil in regions of smoothness (thus providing high accuracy), and that degrades automatically to lower order polynomials when a discontinuity is detected within the large stencil. The lower order polynomials are based on smaller stencils, so that they may avoid the discontinuity. The high order polynomial is never actually computed, but high accuracy is in fact obtained by blending the lower order polynomials with carefully designed non  linear weights that reproduce the value that would be given, at one particular point, by the high order polynomial. The high order optimal polynomial is thus replicated only at one point at a time. If the reconstruction is needed at several points, as in the quadratures required by the integration of \eqref{eq:TheSemidis}, then several reconstruction steps must be computed, each time with different weights.

This problem is particularly severe in balance laws, such as the shallow water equation, where one needs to evaluate the source at quadrature points in the interior of the cell.
For example, optimal weights for the cell centre do not exist for $\WENO$ constructions of order $4k-1$ for any integer $k\geq1$ (e.g. $\WENO3$), 
and  are not always in $[0,1]$ for $\WENO$ constructions of order $4k+1$ (e.g. $\WENO5$), 
see \cite[p. 194]{QiuShu:02}. 
There is a technique to treat negative weights \cite{ShiHuShu:2002}, but it requires to compute two different reconstructions per point. Also, the evaluation of a posteriori error indicators may require to compute accurate quadratures of some form of the local residual, as in the the case of the indicator based on the entropy production, see 
\cite{P:entropy,PS:entropy,PS:shentropy}. Here too, the possibility of computing cheaply the reconstruction at interior nodes is crucial.

Moreover, in non-uniform grids, WENO weights depend on the mesh geometry. For example, in 1D, the weights  depend on the ratio of the sizes of the neighbour and of the current cell, see e.g. \cite{WangFengSpiteri,PS:shentropy}, and they additionally also depend on the disposition of the neighbouring cells in 2D \cite{HuShu:1999,DK:2007:linear,PS:HYP12}. 

A source of non uniform grids typically is mesh adaptivity of $h$-type
or moving mesh algorithms.
Both these techniques need the spatial reconstruction for time advancement, but also in order to perform another important task. In fact, they both involve a change in the mesh that occurs after the conclusion of each time step. In these cases, it is necessary to project the solution from one grid to the new mesh produced by the adaptive algorithm. The  cells of the new grid are subcells of the previous ones in the case of h-AMR (see e.g. \cite{SCR:16,pyClaw:paper}) while they lie in more general positions in the case of moving mesh methods (see e.g.\cite{TangTang:2003}). For schemes of order at least $3$, one must be able to compute the subcell averages with the same accuracy of the scheme and this requires   reconstructions at inner quadrature points, see e.g. \cite{SCR:16}.

Other schemes for which these reconstructions can be of interest are the $\mathrm{P_NP_M}$ schemes of \cite{Dumbser:2008:PnPm} in which at each step a reconstruction from cell averages is required to compute a reliable reconstructing polynomial  inside each cell. Here one needs the functional expression of the polynomial and not just its point values.

\paragraph{Summary.}
The first instance in which the need to have an expression for the reconstruction polynomial was answered in $\WENO$-type constructions, was in the third order central scheme of \cite{LPR:99}. There the authors introduced a new reconstruction procedure of order three.
In this paper we extend this idea to a general technique to obtain a high order, essentially non-oscillatory, interpolation polynomial that is globally defined in the whole cell 
(\S\ref{sec:CWENO}).

The new reconstructions are based on an optimal polynomial defined on a central large stencil and on a set of lower degree polynomials defined on sub-stencils of the bigger one.
The selection mechanisms of the polynomials actually employed to compute the reconstruction is similar to the $\WENO$ one (reviewed in \S\ref{sec:WENO}), but it includes an extra polynomial of the same degree of the optimal one.
 For this reason, following \cite{LPR:99}, we call the reconstructions Central WENO ($\CWENO$).
The main difference between WENO and CWENO is that the latter does not compute reconstructed values at given points in the cell but rather a reconstruction polynomial defined in the whole cell. 

The convergence rates of the $\CWENO$ reconstructions, when the 
Jiang-Shu smoothness indicators of \cite{JiangShu:96} are employed, depends 
on the value chosen for the small parameter $\epsilon$ appearing in the algorithm.
This value must be chosen carefully due to the behaviour of the smoothness indicators close to local extrema and this issue is thus present in the $\WENO$ setting as well. Many techniques were proposed to overcome this difficulty in the $\WENO$ framework, \cite{HAP:2005:mappedWENO,FengHuangWang:14,DB:2013,Arandiga}. 
The technique of \cite{Arandiga}, consisting in choosing a value for $\epsilon$ as a function of the mesh size, was extended to the $\CWENO$ setting, at order 3, by \cite{Kolb2014} on uniform grids and by \cite{CS:epsweno} on a non-uniform mesh. In \S\ref{sec:anal:smooth} we show that the choices $\epsilon\sim h^2$ and $\epsilon\sim h$ guarantee the optimal convergence rate for a $\CWENO$ construction of any order, under the condition that no polynomial involved in the reconstruction is of degree smaller than one half of the degree of the optimal polynomial.

The essentially non-oscillatory behaviour of $\CWENO$ when the data to be interpolated contain a discontinuity is, from a practical point of view, very similar to that of $\WENO$. However, from a theoretical point of view, the situation is quite different, due to the employment of the extra candidate polynomial of high degree. In \S\ref{sec:anal:disc} we introduce a condition (that we call {\em Property R}) that must be satisfied by this extra high degree candidate polynomial in order to ensure that the reconstruction has essentially non-oscillatory properties. 
Furthermore, we show that this property is satisfied by all the one-dimensional $\CWENO$ constructions of any order.

Finally, in \S\ref{sec:numerical} we provide extensive numerical evidence of the accuracy and non-oscillatory behaviour of $\CWENO$ constructions of order up to 9. Furthermore, in order to test the reconstruction at points in the interior of the computational cells, we show applications to the Euler gas dynamics equation with source terms and to the development of well-balanced schemes for the shallow water equation.

\section{A review of WENO reconstructions} \label{sec:WENO}
Before introducing the $\CWENO$ class of reconstructions, we briefly review the $\WENO$ one.
Fixing a stencil $\{\Omega_{j-\grado},\ldots,\Omega_{j+\grado}\}$, 
the definition of $\Precj$ that maximises the accuracy for smooth functions $u(x)$ is clearly the polynomial $\Popt$ of degree $G=2\grado$ which interpolates the $2\grado+1$ cell averages $\ca{u}_{j-\grado},\ldots,\ca{u}_{j+\grado}$,  
which is easily computed following \cite{Shu97}.
Obviously, such a polynomial might be very oscillatory if a jump discontinuity is present in the stencil. In view of this, $\WENO$ never computes $\Popt$ directly, but makes instead a clever use of all the polynomials of lower degree ($\grado$) whose stencil avoids the discontinuity.

\begin{definition}
Fix a point $\hat{x}\in\Omega_j$ and an integer $\grado$.
The $\WENO$ reconstruction operator is given by
\[
R_j(\hat{x})
=
\WENO(P_1,\ldots,P_{\grado+1};\Popt,\hat{x}) \in \R,
\]
where the $P_k$'s, $k=1,\dots,\grado+1$ are polynomials of degree $\grado$, $\Popt$ is a polynomial of degree $G=2\grado$ which guarantees the required accuracy $2\grado+1$.
The point value $R_j(\hat{x})$  is computed as follows:
\begin{enumerate}
\item First, find a set of coefficients $d_1(\hat{x}),\ldots,d_{\grado+1}(\hat{x})$  such  that
\[ \sum_{k=1}^{\grado+1} d_k(\hat{x})P_k(\hat{x}) = \Popt(\hat{x}) \quad
\text{ and } \quad 
\sum_{k=1}^{\grado+1} d_k(\hat{x})=1.
\]
These will be called optimal or linear coefficients.
\item Then nonlinear coefficients $\omega_k$ are computed from the optimal (or linear) ones as
\begin{equation} \label{eq:WENOweights}
\alpha_k(\hat{x}) = \frac{d_k(\hat{x})}{(I[P_k]+\epsilon)^t}
\qquad
\omega_k(\hat{x}) = \frac{\alpha_k(\hat{x})}{\sum_{i=1}^{\grado+1}\alpha_i(\hat{x})},
\end{equation}
where $I[P_k]$ denotes a suitable regularity indicator (to be defined later) evaluated on the polynomial $P_k$, $\epsilon$ is a small positive quantity and $t\geq2$.
\item Finally
\begin{equation}
\Precj(\hat{x}) = \sum\limits_{k=1}^{\grado+1} \omega_k(\hat{x}) P_k(\hat{x}) 
\end{equation}
\end{enumerate}
\end{definition}

The regularity indicators should measure the ``smoothness'' of the polynomial $P_k$ on the computational cell $\Omega_j$.
A regularity indicator is a positive semi-definite operator from $\mathbb{P}$
to $\mathbb{R}^+$, which typically depends on the derivatives of the polynomial in order to detect its oscillatory behaviour. The classical example is the Jiang-Shu indicator, defined in \cite{JiangShu:96} as
\begin{equation}  
\label{eq:JiangShuInd} 
I[P] = \sum_{l\geq 1} \diam(\Omega)^{2l-1} \int_{\Omega} \left(\tfrac{\mathrm{d}^l}{\mathrm{d}x^l} P(x)\right)^2 \dx
\end{equation}
Note that the summation is in fact finite, and that on smooth data $I[P]=\Ogrande(\diam(\Omega)^2)$ at most. In this work we will employ the Jiang-Shu indicators, but other possibilities were explored in \cite{DB:2013,HKLY:2013}. 

We record here an useful property of these indicators.
In what follows, $h$ will denote $\diam(\Omega)$ for a generic cell in the grid.

\begin{remark} \label{rem:IndLip}
The Jiang-Shu indicator of a polynomial $P$ is Lipschitz continuous with respect to the cell averages $\ca{u}_{j-r},\ldots,\ca{u}_{j+s}$, with $r$ and $s$ positive integers, interpolated by $P$.
In fact, $P$ depends linearly on the data and thus $I_P$ is a positive semi-definite quadratic form with respect to $\ca{u}_{j-r},\ldots,\ca{u}_{j+s}$.
\end{remark}

\begin{summary} \label{summary:WENO}
The ingredients of the success of the $\WENO$ reconstruction are the following.
\begin{enumerate}
\item \label{request:I}
The regularity indicators \eqref{eq:JiangShuInd}, which are of size $\Ogrande(h^2)$ on regular data, but $I[P]\asymp1$ in the case of discontinuous data. With $f(h)\asymp g(h)$ (for $h\to 0$) we mean that the limit of $f(h)/g(h)$ exists, is finite and not zero.
\item \label{request:smooth}
 Thanks to the definition of the nonlinear weights, 
the reconstruction error at point $\hat{x}$ is given  by
\begin{equation}
\label{eq:recerror:Dd}
\begin{aligned}
u(\hat{x})- R_j(\hat{x}) &=
u(\hat{x})- \Popt(\hat{x}) + \sum_{k=1}^{g+1} \big(d_k(\hat{x})-\omega_k(\hat{x})\big)P_k(\hat{x})
\\
&=\underbrace{(u(\hat{x})-\Popt(\hat{x}))}_{\Ogrande(h^{2\grado+1})}
+ \sum_{k=1}^{g+1} 
\big(d_k(\hat{x})-\omega_k(\hat{x})\big)
\underbrace{(P_k(\hat{x})-u(\hat{x}))}_{\Ogrande(h^{\grado+1})} 
\end{aligned}
\end{equation}
where the last equality is true since $\sum_{k=1}^{g+1}d_k=\sum_{k=1}^{g+1}\omega_k=1$. 
From the above formula it is clear that the accuracy of the $\WENO$ reconstruction equals the accuracy of  $\Popt$ only if $d_k-\omega_k=\Ogrande(h^{\grado})$ in the case of smooth data. This is ensured by the regularity of the smoothness indicators and by an appropriate choice of the parameter $\epsilon$ (see \cite{Arandiga,CS:epsweno}).
\item \label{request:disc}
In the case of discontinuous data, suppose first that there is one smooth substencil, so that at least one of the regularity indicators is $\Ogrande(h^2)$. Then, the normalisation procedure in \eqref{eq:WENOweights} ensures that
for all $k$ such that $I[P_k]\asymp1$, then $\omega_k\simeq0$. In this way, only the $P_k$'s with $I[P_k]=\Ogrande(h^2)$ contribute to the reconstruction. This is the case provided there is one singularity in the stencil, which does not occur in the central cell.
\item \label{request:discCx} On the other hand, if the discontinuity is in the central cell, each $I[P_k]\asymp1$. In the case of finite differences (see \cite{Shu97,HOEC:1986}) one can prove that each candidate polynomial is monotone in the central cell and thus deduce that the reconstructed value will not increase the total variation. In the case of finite volumes, instead, the reconstructed  data is not guaranteed to satisfy Total Variation Diminishing (TVD) bounds, although typically spurious oscillations are not observed.
\end{enumerate}
\end{summary}

For example, for reconstructions from point values applied to the case of Heaviside data,
all candidate polynomials are bounded by the values before and after the jump, see \cite[p. 347]{Shu97}.
The reconstruction is then total variation bounded for the case of Heaviside data with a Lipschitz perturbation, see \cite[Theor 4.1, p. 359]{HOEC:1986}.


This procedure is  extremely successful and allowed to construct very high order essentially non-oscillatory schemes (see \cite{Shu:2009:WENOreview} and references therein), but it has a few shortcomings.
The linear coefficients $d_k(\hat{x})$ depend explicitly on the location of $\hat{x}$ inside the cell $\Omega_j$. (Their values have been tabulated for the cell boundaries in one space dimension for uniform grids \cite{Shu97, Arandiga}). 
In order to construct a finite volume scheme, the computation of linear and nonlinear weights is required at different points on the cell boundary: two points in one space dimension and at least six (on triangles) and 8 on a Cartesian mesh for a scheme of order at least three in two space dimensions. Even more reconstructions are needed for balance laws, where the cell average of the source has to be evaluated, and for higher dimensions.

Moreover, for interior points, the linear coefficients may not exist (e.g. WENO3 at cell centre) or be non-positive (e.g. WENO5 at cell centre). Results on the existence of $d_k(\hat{x})$ for general $\hat{x}$ have been proven for example in \cite{Carlini:06, Gerolymos:12}. A procedure to circumvent the appearance of negative weights was proposed in \cite{ShiHuShu:2002}.

From the next section, we study the $\CWENO$  schemes which are not affected by any of these troubles, since the linear coefficients are not needed to guarantee the accuracy of the reconstruction in smooth cases. Thus they can be chosen rather arbitrarily and be the same for every reconstruction point in the cell. An additional advantage is that the computation of the $\alpha_k$ and the $\omega_k$ is performed only once per cell and not once per reconstruction point.

\section{The CWENO operator} \label{sec:CWENO}
In this section we introduce a general framework for defining and studying $\CWENO$ reconstructions, which encompasses the one of \cite{LPR:99} and all variations published later in one and more space dimensions, on structured and unstructured grids.
Moreover, this will allow us to propose higher order extensions.

\begin{definition} \label{def:CWENO}
Consider a set of data (point values or cell averages) and a polynomial $\Popt$ of degree $G$, which interpolates in some sense all the given data ({\em optimal polynomial}).
The $\CWENO$ operator computes a reconstruction polynomial
\[\Prec = \CWENO(\Popt,P_1,\ldots,P_{\emme}) \in \Poly{G}\]
from  $\Popt\in\Poly{G}$ and a set of lower order
alternative polynomials $P_1,\ldots,P_{\emme} \in\Poly{g}$, where $g<G$ and $\emme\geq1$. 
The definition of $\Prec$ depends on the choice of a set of positive real 
coefficients $d_0,\ldots,d_{\emme}\in[0,1]$ such that $\sum_{k=0}^{\emme} d_k=1$, $d_0\neq0$ (called {\em linear coefficients}) as follows:
\begin{enumerate}
\item first, introduce the polynomial $P_0$ defined as
\begin{equation}
P_0(x) = \frac{1}{d_0}\left(\Popt(x)-\sum_{k=1}^{\emme}d_kP_k(x)\right) \in\Poly{G} 
\end{equation}
\item then the nonlinear coefficients $\omega_k$ are computed from the linear ones as
\begin{equation} \label{eq:OmegaFromD}
\alpha_k = \frac{d_k}{(I[P_k]+\epsilon)^t}
\qquad
\omega_k = \frac{\alpha_k}{\sum_{i=0}^{\emme}\alpha_i},
\end{equation}
where $I[P_k]$ denotes a suitable regularity indicator (e.g. the Jiang-Shu ones of eq. \eqref{eq:JiangShuInd}) evaluated on the polynomial $P_k$, $\epsilon$ is a small positive quantity and $t\geq2$;
\item and finally
\begin{equation}
\Prec(x) = \sum\limits_{k=0}^{\emme} \omega_k P_k(x) \in\Poly{G}
.
\end{equation}
\end{enumerate}
\end{definition}

Note that the polynomial $P_0\in\Poly{G}$ {\em is} part of the reconstruction, that $\CWENO$ provides a polynomial $\Prec$ that can be evaluated at any point within the cell, and that all coefficients $\omega_k$ 
involved in the reconstruction {\em do not} depend on the particular points where the reconstruction is needed.

\begin{remark}
In the case of reconstruction from cell averages, from the definition, it is trivial to check that, if all candidate polynomials satisfy the  conservation property
\[ \tfrac{1}{|\Omega|} \int_{\Omega}\Popt\dx = \tfrac{1}{|\Omega|} \int_{\Omega}P_k\dx = 
\ca{u}_{\Omega}
\]
for $k=1,\ldots,\emme$, then also $P_0$ and $\Prec$ have the same cell average:
\[ \tfrac{1}{|\Omega|} \int_{\Omega}P_0\dx = \tfrac{1}{|\Omega|} \int_{\Omega}\Prec\dx 
=\ca{u}_{\Omega}
.
\]
\end{remark}

\begin{remark}
The previous definitions may be cast in either one-dimensional or multi-dimensional settings. In the latter case $x=(x_1,\ldots,x_n)\in\R^n$ and $\Poly{g}$ denotes the space of polynomials in $n$ variables with degree at most $g$.
\end{remark}

Typically, in Finite Volume schemes, the optimal polynomial $\Popt$ is taken to be the polynomial interpolating all the data in the stencil of the reconstruction in the sense of cell averages.
For example in one space dimension, in each cell $\Omega_j$, the original CWENO construction of \cite{LPR:99}, is a third order accurate $\CWENO$ procedure with $\emme=2$, $\Popt=P^{(2)}$ the parabola defined on the centred 3-cell stencil $\Omega_{j-1},\Omega_j,\Omega_{j+1}$, and $P_1=P^{(1)}_L$, $P_2=P^{(1)}_R$ being the two linear polynomials interpolating the data in $\Omega_{j-1},\Omega_j$ and $\Omega_j,\Omega_{j+1}$ respectively. 
The same reconstruction was recently considered in a non-uniform mesh setting in \cite{PS:shentropy,CS:epsweno}. 

A fifth order version $\CWENO(P^{(4)},P^{(2)}_L,P^{(2)}_C,P^{(2)}_R)$ was proposed in \cite{Capdeville:08}, using a centred fourth degree polynomial interpolating the data in $\Omega_{j-2},\ldots,\Omega_{j+2}$ 
and the same three parabolas employed in the classical WENO5 scheme, namely those interpolating the data in $\Omega_{j-2+r},\Omega_{j-1+r},\Omega_{j+r}$ for $r=0,1,2$ respectively.

Along the same lines, in this paper we will introduce a seventh order reconstruction $\CWENO7=\CWENO(P^{(6)},P^{(3)}_{LL},P^{(3)}_L,P^{(3)}_R,P^{(3)}_{RR})$, where the optimal polynomial is the sixth order $P^{(6)}=\Popt$ interpolating the data in $\Omega_{j-3},\ldots,\Omega_{j+3}$ and $P_1=P^{(3)}_{LL}$, $P_2=P^{(3)}_L$, $P_3=P^{(3)}_R$, $P_4=P^{(3)}_{RR}$ are the third order polynomials interpolating $\ca{u}_{j-3+r},\ldots,\ca{u}_{j+r}$ for $r=0,1,2,3$.

Similarly, we will also propose the ninth order reconstruction $\CWENO9$ with $\emme=5$, $\Popt$ the eight order polynomial interpolating the data in $\Omega_{j-4},\ldots,\Omega_{j+4}$ and $P_1,\ldots,P_5$ are fourth order polynomials interpolating $\ca{u}_{j-4+r},\ldots,\ca{u}_{j+r}$ for $r=0,1,2,3,4$.

A few two-dimensional $\CWENO$ reconstructions can be found in the literature, including those of \cite{LPR:2001} where this technique was proposed and \cite{SCR:16} where it is generalised to non globally Cartesian grids.

\begin{remark}
We note that the coefficients $d_k$ appearing in Definition \ref{def:CWENO} do not need to satisfy accuracy requirements and they can be thus arbitrarily chosen, provided that they are positive and add up to $1$. A possible choice of coefficients is described just below.
\end{remark}

We start assigning weights to the low degree polynomials, biasing towards the central ones, because they would yield a smaller interpolation error. A reconstruction of order $2g+1$ is composed of $\emme=\grado+1$ polynomials of degree $\grado$. These are the $\emme$ polynomials which would compose a WENO reconstruction of the same order. Let $j=1,\dots,\emme$ be the indices of the low degree polynomials. We start computing temporary weights
\begin{equation} \label{eq:temporary:d}
\tilde{d}_j=\tilde{d}_{\emme+1-j}=j, \qquad \mbox{for}\; 1\leq j \leq \frac{\emme+1}{2}.
\end{equation}
Then we choose the linear coefficient $d_0\in(0,1)$ of the high order polynomial $P_0$. 
The final weights are given by
\[
d_j= \frac{\tilde{d}_j}{\sum_{i=1}^{\emme} \tilde{d}_i}(1-d_0).
\]

The value of $d_0$ must be bounded away from $0$ and from $1$. In fact, when $d_0$ is too close to $0$ the polynomial $P_0$ becomes unbounded. On the other hand, when $d_0$ is close to $1$, the reconstruction polynomial $\Prec$ will almost coincide with $\Popt$, irrespectively of the oscillation indicators.

In this paper we will mainly consider the two cases $d_0=\tfrac12$ and $d_0=\tfrac34$.
For instance, for $\CWENO5$, and $d_0=\tfrac34$, we have the left and right parabola with weight $d_1=d_3=\tfrac{1}{16}$ and $d_2=\tfrac18$.

\subsection{Implementation of the reconstruction in 1D}
The main task for computing a $\CWENO$ reconstruction efficiently is to optimise the computation of the coefficients of the interpolating polynomials.
In $\WENO$ the reconstruction is computed only at one point at a time and thus the Lagrange form of the interpolating polynomials is well suited to the task, see \cite{Shu97}. In contrast, here we need the functional representation of the polynomials and therefore it is more convenient to start from the Newton basis and finally get the representation of the polynomials in the basis of the monomials for the computation of the smoothness indicators.

Recall that $\ca{u}_j$ denotes the cell average of $u(x)$ on the generic cell $\Omega_j$ of the grid, which has size $h_j$. 
In order to compute the $\CWENO$ reconstruction in the $j$-th cell,
we need the explicit expression of the polynomial 
of degree $k$ that interpolates the cell averages $\ca{u}_{j-r},\ldots,\ca{u}_{j-r+k}$. Here $r$ denotes the offset of the stencil with respect to the $j$-th cell. Note that for a typical $\CWENO$ reconstruction one needs $\grado+1$ polynomials of degree $\grado$ with $r=0,\ldots,\grado$ and a polynomial of degree $G=2\grado$ with offset $r=\grado$. Note also that $\grado$ out of the $\grado+1$ polynomials of degree $\grado$ employed in the reconstruction for cell $\Omega_j$ are used also for the reconstruction in the cell $\Omega_{j+1}$, so that one needs to compute only one new polynomial per cell. 

It is thus convenient to compute all divided differences of the set of cell averages as a preprocessing stage to the computation of the reconstruction. In particular, denote the {\em divided differences} of the cell averages by
\begin{equation} \label{eq:divdiff}
\tilde{\delta}_{j,1} = \ca{u}_j,
\qquad
\tilde{\delta}_{j,p} = \frac{\tilde{\delta}_{j+1,p-1}-\tilde{\delta}_{j,p-1}}{\sum_{i=j}^{j+p-1}h_i}
\text{ for } p>1.
\end{equation}
For later convenience, let us introduce also the {\em undivided differences}
\begin{equation}\label{eq:undivdiff}
\delta_{j,p} = \left.\tilde{\delta}_{j,p}\right|_{\forall i: h_i=1}
,
\end{equation}
which are useful for computations on uniform grids.

Following \cite{Shu97} we note that a polynomial $p(x)$ of degree $k$ interpolating a set of consecutive cell averages can be easily computed by differentiating the polynomial $q(x)$ of degree $k+1$ that interpolate the quantities $S_i=\sum_{l\leq i}h_l\ca{u}_l$ in the interpolation nodes  $x_{i}+h_i/2$. It is easy to see that, for the sake of computing $p(x)$, the zero-th order term in $q(x)$ is not relevant. Thus the only divided differences that are needed are the ones listed in \eqref{eq:divdiff}.

From now on, let us focus on a reference cell $j=0$ and assume that its cell centre is at  $x_0=0$.
Let $p^{(k)}_r(x)$ be the degree $k$ polynomial with stencil offset $r$.
Applying the Newton interpolation, one finds that its primitive is
\begin{equation} \label{eq:q:Newton}
q^{(k+1)}_r(x) = \sum_{i=1}^{k+1} \tilde{\delta}_{-r,i} \prod_{l=0}^{i-1}(x-x_{-r-\um+l})
+\text{constant term} 
\end{equation}
and we write it in the basis of the monomials as
\begin{equation} \label{eq:q:monomial}
q^{(k+1)}_r(x) = 
\sum_{i=1}^{k+1} \tilde{\delta}_{-r,i} \sum_{m=1}^{i} \tilde{\gamma}^k_{r,i,m}x^m
+\text{constant term} 
\end{equation}
where $\tilde{\gamma}^k_{r,i,m}$ is the weight of the divided difference of order $i$ and offset $-r$ (i.e. $\tilde{\delta}_{-r,i}$) appearing into the coefficient of the monomial $x^m$.
Note that only the coefficients $\tilde{\gamma}^k_{r,i,m}$ for $m>0$ appear in the derivative of $q^{(k+1)}_r(x)$.
 By direct comparison of the last two equations one finds for the linear term that
\[
\tilde{\gamma}^k_{r,1,1}=1 ,
\qquad
\tilde{\gamma}^k_{r,i,1}
=(-1)^{i-1} \sum_{n=0}^{i-1}\prod_{\substack{l=0,\ldots,i-1\\l\neq n}} x_{l-r-\um}
,\;i>1
\]
and in general that
\begin{equation} \label{eq:tildegamma}
\begin{aligned}
\tilde{\gamma}^k_{r,i,m}
&=(-1)^{i-m}
\sum_{n_1=0}^{i-1}
\sum_{n_2=n_1+1}^{i-1}
\!\cdots \!\!\!\!\!
\sum_{n_m=n_{m-1}+1}^{i-1} 
\prod_{\substack{l=0,\ldots,i-1\\l\neq n_1,\ldots,n_m}} x_{l-r-\um} ,
\quad m<i
\\
\tilde{\gamma}^k_{r,m,m} &= 1,
\\
\tilde{\gamma}^k_{r,i,m} &= 0, \quad m>i
.
\end{aligned}
\end{equation}
%
%
%
Finally, the sought polynomial $p^{(k)}_r$ is found differentiating $q^{(k+1)}_r$:
\begin{equation} \label{eq:p:monomial}
p^{(k)}_r(x) = 
\sum_{i=1}^{k+1} \tilde{\delta}_{-r,i} 
  \sum_{m=1}^{i} \tilde{\Gamma}^k_{r,i,m}x^{m-1},
  \qquad
  \tilde{\Gamma}^k_{r,i,m} = m \tilde{\gamma}^k_{r,i,m}
.
\end{equation}
Note in particular that the values of $\tilde{\gamma}^k_{r,i,0}$ are not needed in the expression for $p^{(k)}_r(x)$.

Note that \eqref{eq:tildegamma} may be rewritten in terms of the cell sizes in the neighbourhood by exploiting the identity 
\[
x_{l-r-\um} = -\sum_{i=l-r}^{-1}h_i
+\sign(l-r)\frac{h_0}{2}
+\sum^{l-r-1}_{i=1}h_i,
\]
in which one of the two summations is always empty, depending on the sign of $l-r$.


Of course considerable simplifications occur on uniform grids, where one can write
\begin{equation} \label{eq:q:Newton:unif}
\begin{aligned}
q^{(k+1)}_r(x) &= \sum_{i=1}^{k+1} \tilde{\delta}_{-r,i}\prod_{l=0}^{i-1}(x-(-r-\um+l)h)
+\text{constant term} 
\\
&= \sum_{i=1}^{k+1} \delta_{-r,i} \prod_{l=0}^{i-1}(\tilde{x}-(-r-\um+l))
+\text{constant term} 
,
\end{aligned}
\end{equation}
where we recall that $\delta_{-r,i}$ are the undivided differences
and we have set $\tilde{x}=x/h$.
The above polynomial can be put in the form \eqref{eq:q:monomial} with
\begin{equation}\label{eq:tildegamma:unif}
\tilde{\gamma}^k_{r,i,m}
=(-h)^{i-m}
\sum_{n_1=0}^{i-1}
\;
\sum_{n_2=n_1+1}^{i-1}
\!\cdots \!\!\!\!\!
\sum_{n_m=n_{m-1}+1}^{i-1} 
\,
\prod_{\substack{l=0,\ldots,i-1\\l\neq n_1,\ldots,n_m}} (l-r-\um)
.
\end{equation}
An alternative  form is
\begin{equation} \label{eq:q:monomial:unif}
q^{(k+1)}_r(x) = 
\sum_{i=1}^{k+1} {\delta}_{-r,i} 
\sum_{m=1}^{i} {\gamma}^k_{r,i,m}x^m
+\text{constant term} 
\end{equation}
with
\begin{equation}\label{eq:tilde:unif}
{\gamma}^k_{r,i,m}
=(-1)^{i-m}
\sum_{n_1=0}^{i-1}
\;
\sum_{n_2=n_1+1}^{i-1}
\!\cdots \!\!\!\!\!
\sum_{n_m=n_{m-1}+1}^{i-1} 
\,
\prod_{\substack{l=0,\ldots,i-1\\l\neq n_1,\ldots,n_m}} (l-r-\um)
.
\end{equation}
Finally,
\begin{equation} \label{eq:p:monomial:unif}
p^{(k)}_r(x) = 
\sum_{i=1}^{k+1} {\delta}_{-r,i} 
  \sum_{m=0}^{i} {\Gamma}^k_{r,i,m}x^{m-1},
  \qquad
  {\Gamma}^k_{r,i,m} = m {\gamma}^k_{r,i,m}
\end{equation}

\begin{table}
$
\Gamma_{3,i,m} = 
\begin{bmatrix} 
1 & &   &   &    &  &   \\  
6 &   2  &  &  &   &   &  \\   
\nicefrac{71}{4} & 15&   3 &   &   &   &  \\ 
22& 43& 24&   4  &  &   &  \\   
-\nicefrac{71}{16}& \nicefrac{45}{2}&\nicefrac{105}{2}& 30&   5 &   &  \\   
\nicefrac{27}{8}&-\nicefrac{341}{8}&-45& 25& 30&   6 &     \\
-\nicefrac{225}{64}&\nicefrac{1813}{16}&\nicefrac{777}{16}&-\nicefrac{245}{2}&-\nicefrac{175}{4}& 21&   7  
\end{bmatrix}
$

$\Gamma_{2,i,m} =
\begin{bmatrix} 
1&&&&\\
4& 2& & &\\
\nicefrac{23}{4}& 9& 3&&\\
-1&7& 12& 4&\\
\nicefrac{9}{16}&-\nicefrac{25}{2}&-\nicefrac{15}{2}& 10& 5
\end{bmatrix}
$

$\Gamma_{1,i,m} =
\begin{bmatrix}
1& &&\\
2& 2&&\\ 
-\nicefrac{1}{4}& 3& 3&\\ 
0&-5& 0& 4
\end{bmatrix}$
\hfill
$\Gamma_{0,i,m} =
\begin{bmatrix}
1&&&\\
 0& 2&&\\ 
-\nicefrac{1}{4}&-3& 3&\\ 
 1& 7&-12& 4
\end{bmatrix}$

\caption{Table of the $\Gamma$ coefficients of \eqref{eq:p:monomial:unif} used in the computation of $\CWENO$ reconstructions up to order $7$ on uniform grids.}
\label{tab:Gamma}
\end{table}

In Table \ref{tab:Gamma} we list the values of the coefficients $\Gamma^k_{r,i,m}$ needed for the $\CWENO$ reconstructions up to order $7$. The coefficients for the higher order cases can be computed using \eqref{eq:tilde:unif} and \eqref{eq:p:monomial:unif}. 

If the final accuracy of the reconstruction is $2\grado+1$, we need the stencil $\Omega_{-\grado},\ldots,\Omega_\grado$. Here we must compute the polynomial $\Popt$ which has offset $\grado$ and contains monomials of degree $m$ up to $2\grado$ and all polynomials of degree $k=\grado$ with offset $r=0,\ldots,\grado$.
Note that the elements of $\Gamma^k_{r,i,m}$ are independent of $k$. 
Therefore they can all be stored in a matrix $\Gamma_{r,i,m}$
and the coefficients needed for the polynomial of degree $k$ with shift $r$ are in the top-left $(k+1)\times(k+1)$ submatrix of the matrix $\Gamma_{r,i,m}$ which are listed in Table  \ref{tab:Gamma}.

For example, for $\CWENO7$, we need all coefficients of $\Gamma_{3,i,m}$ reported in the table to build $\Popt$ and also the top $4\times 4$ submatrices from each $\Gamma_{r,i,m}$ (including $\Gamma_{3,i,m}$) to build the coefficients of the four cubic polynomials which compose the reconstruction.

\section{Analysis of the $\CWENO$ reconstruction in the smooth case}
\label{sec:anal:smooth}

This topic corresponds to point \ref{request:smooth} in the list of Summary \ref{summary:WENO}.
In order to perform the analysis of the $\CWENO$ reconstruction, let us focus on a fixed computational cell
$\Omega_0$ and assume that its cell centre is $x_0=0$. The $\CWENO$ procedure will be applied to the set of exact cell averages $\ca{u}_j$ of a given function $u(x)$. 
Let us assume that $\Popt\in\Poly{G}$ interpolates the cell averages of $\Omega_0$ and of a suitable number of neighbours, so that its approximation order is $\Ogrande(h^{G+1})$, if the function $u(x)$ is sufficiently regular. Furthermore the polynomials $P_r\in\Poly{\grado}$ are typically chosen to interpolate $\grado+1<G+1$ cell averages inside the stencil of $\Popt$ and their approximation order is $\Ogrande(h^{\grado+1})$. The reconstruction error at a point $x\in\Omega_0$ is thus given  by
\begin{equation}
\begin{aligned}
u(x)- \Prec(x) &=
u(x)- \Popt(x) + \sum_{r=0}^{\emme} (d_r-\omega_r)P_r(x)
\\
&=\underbrace{(u(x)-\Popt(x))}_{\Ogrande(h^{G+1})}
+ \sum_{r=0}^{\emme} 
(d_r-\omega_r)
\underbrace{(P_r(x)-u(x))}_{\Ogrande(h^{\grado+1})} 
\end{aligned}
\end{equation}
where the last equality is true since $\sum_{r=0}^{\emme}d_r=\sum_{r=0}^{\emme}\omega_r=1$. 
From the above formula it is then clear that the accuracy of the $\CWENO$ reconstruction equals the accuracy of its first argument $\Popt$ only if $(d_r-\omega_r)=\Ogrande(h^{G-\grado})$ in the case of smooth data, as in standard $\WENO$.

As we will see, $\CWENO$, exactly as $\WENO$, can be influenced by the chosen value chosen for $\epsilon$ in \eqref{eq:WENOweights} and \eqref{eq:OmegaFromD}. While obviously a value that is too large will promote the onset of spurious oscillations, a value that is too small may induce a degradation of the convergence order close to local extrema. This effect was first noticed in the $\WENO$ setting in \cite{HAP:2005:mappedWENO} and a technique consisting in a post-processing of the $\WENO$ weights known as $\WENO$-M was proposed in the same paper and later extended to higher order in \cite{FengHuangWang:14}. Another approach involving additional smoothness indicators, known as $\WENO$-Z has also been studied (see \cite{DB:2013} and references therein). In \cite{Arandiga} the authors devise a way to relate the value of $\epsilon$ to the mesh size in order to guarantee the correct convergence order and this technique has been extended to the $\CWENO$ setting in \cite{Kolb2014} for uniform meshes and exploited also on non-uniform meshes in one and two space dimensions, \cite{CS:epsweno} and \cite{SCR:16}, respectively.

For this reason we are mainly interested in the choice
\begin{equation} \label{eq:eps}
\epsilon = \hat{\epsilon}h^p
,
\text{ for } p=1,2
\end{equation}
where $h$ is the mesh size.

We state first a general result on the accuracy of the polynomial $P_0$ computed in step 1 of the $\CWENO$ reconstruction.
\begin{remark} \label{rem:P0accuracy}
$P_0$ is of degree $G$, but its accuracy is $g$:
\begin{align*}
P_0(x)-u(x) &= \frac{1}{d_0}\left[\Popt(x)-\sum_{r\geq1}d_rP_r(x) -d_0u(x)\right]
\\
&=\frac{1}{d_0}\left[\Popt(x)-\sum_{r\geq1}d_rP_r(x) -\left(1-\sum_{r\geq1}d_r\right)u(x)\right]
\\
&=\frac{1}{d_0}\bigg(\Popt(x)-u(x)\bigg)
  + \frac{\sum_{r\geq1}d_r}{d_0}\bigg(u(x)-P_r(x)\bigg)
.
\end{align*}
Thus the accuracy of $P_0$ will coincide with the smallest accuracy of the $P_r$'s.
\end{remark}

In order to prove that the accuracy of $\CWENO$ is $\Ogrande(h^{G+1})$ on smooth data, one has to show that $\omega_r- d_r$ is at least $\Ogrande(h^{G-g})$. This study can be performed extending to our case the technique introduced by \cite{Arandiga} in the case of $\WENO$ and which allows to rewrite $\omega_r-d_r$ in terms of differences among the indicators of the candidate polynomials. 

\begin{proposition}
The $\CWENO$ reconstruction with $\Popt$ of degree $G$ and $P_1,\ldots,P_{\emme}$ of degree $\grado$ is $G+1$ order accurate on smooth solutions, provided that $G \le  2\grado$ and 
$\epsilon=\hat{\epsilon}h^p$ with $p=1,2$.
\end{proposition}
\begin{proof}
The $\CWENO$ procedure starts by computing
\begin{equation} \label{eq:alpha:r}
\begin{aligned}
\alpha_{0} & = \frac{d_{0}}{(\epsilon + I[P_0])^t} \\
\alpha_{r} & = \frac{d_{r}}{(\epsilon + I[P_r])^t} =
\frac{d_{r}}{(\epsilon + I[P_0])^t} 
\left[ 1 + 
\frac{I[P_0] - I[P_r]}{ \epsilon  + I[P_r]} 
\sum_{s=0}^{t-1}
\left( \frac{ \epsilon + I[P_0]}{ \epsilon + I[P_r]} \right)^s \right]
, r=1,\ldots,\emme \\
\end{aligned}
\end{equation}

In order to proceed, we need the Taylor expansions of the differences between the indicators $I[P_r]$ for $r=0,\ldots,\emme$ and we focus on the classical Jiang-Shu indicators of \eqref{eq:JiangShuInd}.
First note that the Jiang-Shu indicator in terms of the  coefficients of a generic polynomial, centred in 0, is given by
\begin{equation} \label{eq:Iai}
 I\left[\sum_{i=0}^{\grado} a_i x^i\right]
 = 
 \sum_{l=1}^\grado 
 \sum_{j=l}^{\grado-1}
 \sum_{\substack{i=j, \\   i+j \ \text{even}}} ^\grado 
	\frac{j!i!}{(j-l)!(i-l)!}
	\frac{2^{2l+1-j-i-\delta_{i,j}}}{j+i-2l+1}
	 a_j a_i h^{j+i}
\end{equation}
where $\delta_{i,j}$ denotes the Kronecker delta.

If the polynomial $\sum_{i=0}^{\grado} a_i x^i$ of degree $\grado$ is interpolating the cell averages of a smooth enough function $u(x)$, then its coefficients satisfy
\begin{equation} \label{eq:ak:interpolating}
a_{i}= \frac{ 1 }{i!} u^{(i)}(0) + \Ogrande(h^{\grado-i+1}), \qquad i=0,1, \cdots, \grado.
\end{equation}
Note  that \eqref{eq:ak:interpolating} holds true also for the polynomial $P_0\in\Poly{G}$, but only for $i=0,\ldots,\grado$. In fact, letting $\Popt=\sum_{i=0}^G b_ix^i$ and $P_r=\sum_{i=0}^g a_{r,i} x^i$ and using the definition of $P_0$, one finds
\begin{equation} \label{eq:P0}
P_0(x) = \sum_{i=0}^G a_{0,i} x^i
= \sum_{i=0}^G \left(\frac{b_i}{d_0} - \sum_{r=1}^{\emme} a_{r,i}\frac{d_r}{d_0}\right) x^i
.
\end{equation}
Next, using \eqref{eq:ak:interpolating} for $\Popt$ and $P_r$ for $r=1,\ldots,\emme$ one gets
\[
a_{0,i}
=
\frac{1}{d_0 i!} 
\left(
\left(1- \sum_{r=1}^{\emme} d_r \right) u^{(i)}(0)
 + \Ogrande(h^{\grado-i+1})
\right)
\]
and finally 
\begin{equation} \label{eq:ak:P0}
a_{0,i}
=\frac{1}{i!} u^{(i)}(0)
 + \Ogrande(h^{\grado-i+1}),
 \qquad i=0,\ldots,\grado.
\end{equation}
It follows that, for $r=0,\ldots,\emme$,
\begin{equation} \label{eq:IPr}
I[P_r]
 = 
 \sum_{l=1}^\grado 
 \sum_{j=l}^{\grado-1}
 \sum_{\substack{i=j, \\ i+j<\grado+2, \\   i+j \ \text{even}}} ^\grado 
	\frac{{j!}}{(j-l)!(i-l)!}
	\frac{2^{2l+1-j-i-\delta_{i,j}}}{j+i-2l+1}
	 u^{(j)}(0) u^{(i)}(0) 
	 h^{j+i}
 +\Ogrande(h^{\grado+2})
.
\end{equation}
We now turn to the terms appearing in \eqref{eq:alpha:r}. 
Recalling \eqref{eq:eps} and
since \eqref{eq:IPr} implies that $I[P_0]-I[P_r]=\Ogrande(h^{g+2})$,
we  have  that
\begin{equation} \label{eq:sum:s}
\sum_{s=0}^{t-1}
\left(
\frac{ \epsilon + I[P_0]}{ \epsilon + I[P_r]} 
\right)^s
=
\sum_{s=0}^{t-1}
\left(
\frac{ \hat\epsilon h^p + I[P_0]}{ \hat\epsilon h^p + I[P_r]} 
\right)^s
=
t + \Ogrande(h^{\grado+2-p}).
\end{equation}
For the terms
\begin{equation}\label{eq:frazione}
\frac{I[P_0] - I[P_r]}{ \epsilon  + I[P_r]}
\end{equation}
we observe that \eqref{eq:IPr} holds true 
for all polynomials involved in the reconstruction and thus for the numerator we have that
\[
I[P_0]-I[P_r] = \Ogrande(h^{\grado+2})
.
\]
Instead, for the denominator of \eqref{eq:frazione}, 
we observe that \eqref{eq:Iai} implies that $I[P_r]=a_1^2h^2+\Ogrande(h^4)$ and,
recalling the choice of $\epsilon$ in \eqref{eq:eps}, we  find
\[ \hat \epsilon h^p + I[P_r] 
 =  Ah^p \left(1 
    +  \sum_{l=1}^\grado 
     \sum_{\substack{j=l\\j\neq p-1} }^{\grado-1}
     \sum_{\substack{i=j, \\   i+j \ \text{even}}} ^\grado
    	\frac{j!i!}{(j-l)!(i-l)!}
    	\frac{2^{2l+1-j-i-\delta_{i,j}}}{j+i-2l+1}
    	 \frac{a_j a_i}{A} h^{j+i-p}
      \right)
\]
where $A= \hat \epsilon $ if $p=1$ and $A= \hat \epsilon + a_1^2$ if $p=2$. 
Now
 $$ 
 \frac{ 1 }{\hat \epsilon h^p + I[P_r]} 
 =  \frac{ 1 }{A h^p} 
    \left( 1 
    +\Ogrande(h^p)
     \right) 
    $$
so that
 $$ \frac{I[P_0] - I[P_r]}{\epsilon + I[P_r]} 
 = \frac{\Ogrande(h^{g+2})}{A h^p}  \left( 1 + \Ogrande(h^p) \right) 
 = \Ogrande(h^{g+2-p}) $$  

Recalling \eqref{eq:sum:s} and \eqref{eq:alpha:r}, we have 
\begin{align*}
\alpha_{r} 
&= 
\frac{1}{(\epsilon + I[P_0])^t} 
\left[ d_r + 
d_{r} 
\frac{I[P_0] - I[P_r]}{ \epsilon  + I[P_r]} 
\sum_{s=0}^{t-1}
\left( \frac{ \epsilon + I[P_0]}{ \epsilon + I[P_r]} \right)^s \right]
\\
&= 
\frac{1}{(\epsilon + I[P_0])^t} 
\left[ d_r + 
\Ogrande(h^{g+2-p})
\left( t+ \Ogrande(h^{g+2-p}) \right)
\right]
\\
&= 
\frac{1}{(\epsilon + I[P_0])^t} 
\left[ d_r + 
\Ogrande(h^{g+2-p}) \right]
\end{align*}
and thus
\[
\left(\sum_{s=0}^{\emme} \alpha_{s} \right)^{-1} 
=(\epsilon + I[P_0])^t
\left[ \sum_{s=0}^{\emme}d_s + 
\Ogrande(h^{g+2-p}) \right]
=(\epsilon + I[P_0])^t
\left[ 1 + 
\Ogrande(h^{g+2-p}) \right]
.
\]
Finally, using \eqref{eq:alpha:r} and the previous relation we have
\begin{equation} \label{eq:omega:r}
\omega_{r} = \frac{\alpha_r}{\sum_{s=0}^{\emme} \alpha_s}
=d_{r}
 \left[
1 +\Ogrande(h^{g+2-p}) \right]
.
\end{equation}
Equation \eqref{eq:omega:r} shows that $\omega_k-d_k = \Ogrande(h^{\grado+2-p})$ and thus the accuracy is maximal provided that $\grado+2-p \ge G-\grado$.
\end{proof}

We point out that, starting from \eqref{eq:ak:interpolating}, all expressions hold in the limit $h\to0$. Obviously, for finite values of $h$, the behaviour of the reconstruction is determined by the relative size of $\hat{\epsilon}h^p$ and  the indicators. Especially in the case $p=0$, when $\hat{\epsilon}$ is too small with respect to $h$, one typically observes a degradation in the convergence rate. On the other hand, if $\hat{\epsilon}$ is too large, one might observe spurious oscillations, since $\hat{\epsilon}$ would override the indicators.

Another case where the size of $\epsilon$ can change the behaviour of the reconstruction is close to a local extremum. It typically happens that the local extremum does not lie in the stencil of all $P_r$'s. Suppose that an extremum is located only in the stencil of $P_{\hat{r}}$ for some $\hat{r}\in\{1,\ldots,\emme\}$. 

In this case a more refined analysis would replace (29) by the Taylor expansions of $I[P_r]$ centred in the middle of the respective stencils and get $I[P_{\hat{r}}]=\Ogrande(h^4)$ while the remaining smoothness indicators would be larger, and this would induce the scheme into selecting only the $\hat{r}$-th stencil, thus degrading accuracy. For this reason, it is important that epsilon is large enough to override the selection of stencils containing extrema, in the smooth case.
For this reason we suggest to employ $\epsilon\approx h^2$ or even $\epsilon\approx h$, as in
\cite{CS:epsweno,SCR:16}.

\section{Analysis in the discontinuous case}
\label{sec:anal:disc}

This section contains a discussion of the behaviour of $\CWENO$ in the case of discontinuous data. While the discussion of the previous section on the smooth case extends partial contributions of previous authors to reconstructions of arbitrary order of accuracy, the discontinuous case, to the best of our knowledge,  has never been analysed in details. In this section we will consider $\CWENO$ as an interpolation algorithm of a known function $u(x)$. We will thus suppose that it is possible to choose the mesh size to ensure that at most one discontinuity is present in the stencil of $\Popt$.

If a discontinuity is present in the stencil of $\Popt$, then the reconstruction is expected to degrade to a combination of the $P_k$'s whose stencil lie in smooth regions. In this respect, the reconstruction behaves as $\WENO$. In the $\WENO$ setting, this fact is almost trivial: only the $P_k$'s contribute to the reconstruction and they are all interpolating polynomials, thus the behaviour of their indicators matches exactly the presence or absence of a discontinuity in the corresponding stencil. 

In the $\CWENO$ setting, the same final result can be proven only if an additional property is verified by the indicators. In fact, in $\CWENO$, also the high order polynomial $P_0$ contributes non trivially to $\Prec$ and thus the behaviour of its indicator should be taken into account as well. However, $P_0$ is not an interpolating polynomial and thus, for the correct behaviour of the reconstruction in the discontinuous case, it is important that the following holds.
\begin{definition}[{\bf Property R}] \label{propertyI0}
We say that a reconstruction $\CWENO(\Popt,P_1,\ldots,P_{\emme})$ satisfies Property R if, whenever a jump-discontinuity is present in the stencil, so that 
$I[\Popt]\asymp 1$ for $h\to 0$, then also $I[P_0]\asymp 1$. 
\end{definition}

We will later prove that Property R holds for all the one-dimensional reconstructions considered in this paper. Here we show a general result of the impact of Property R 
on the behaviour of $\CWENO$ on discontinuous data.

\begin{theorem}\label{th:CWENOdisc}
Assume that Property R holds true for a $\CWENO$ procedure and that $\epsilon=\Ogrande(h)$. If the reconstruction is applied to discontinuous data, but at least one of $I[P_1],\ldots,I[P_{\emme}]$ is of size $\Ogrande(h^2)$, then $\omega_k\sim0$ for every $k\in\{0,\ldots,\emme\}$ such that $I[P_k]\asymp 1$.
\end{theorem}
\begin{proof}
Since the data are discontinuous, then $I[\Popt]\asymp 1$ and, thanks to Property R, also  $I[P_0]\asymp 1$.
Let $K$ be the set $\{k:I[P_k]\asymp 1\}$.
Then the hypothesis guarantees that there exists at least one $l\not\in K$ for which $I[P_l]=\Ogrande(h^2)$.
Therefore $\alpha_l$ is at least of magnitude $h^{-2}$ and thus from \eqref{eq:OmegaFromD} we find that $\omega_l\asymp1$ and  $\omega_k=\Ogrande(h^2)$ for every $k\in K$.
%
\end{proof}

As a corollary, provided that at least one of $P_1,\ldots,P_{\emme}$ insists on a smooth stencil, the reconstruction degrades to a combination of the $P_k$'s insisting on smooth stencils and thus will be Essentially Not Oscillatory. With reference to Summary \ref{summary:WENO}, Property R 
corresponds to point \ref{request:I} and Theorem \ref{th:CWENOdisc} to point \ref{request:disc}.

Notice that Property R 
is not trivial, despite the fact that $P_0$ is a convex combination of the interpolating polynomials $\Popt$ and of all the $P_k$'s.
In fact, at least for the Jiang-Shu indicators, the square inside the integrals in equation \eqref{eq:JiangShuInd} mixes in a nonlinear way the contributions of all the polynomials involved.
For example, consider $\emme=1$, where we have $P_0=\alpha \Popt+(1-\alpha)P_1$ (for $\alpha=1/d_0$) and
\begin{equation}\label{eq:IndOfConvexComb}
I[P_0] = 
\alpha^2 I[\Popt] 
+(1-\alpha)^2 I_{P_1} 
+\alpha(1-\alpha) \sum_{l\geq1} h^{2l-1} 
\int_{\Omega} \left(\tfrac{\mathrm{d}^l}{\mathrm{d}x^l}\Popt\right)  
\left(\tfrac{\mathrm{d}^l}{\mathrm{d}x^l}P_1\right)\dx.
\end{equation}
In the formula above, $I[\Popt]$ and $I[P_1]$ are always non-negative, but there is no way to control the sign of the cross terms. 

We start by showing direct computations regarding property R for the third order $\CWENO$ reconstruction of \cite{LPR:99}, but for generic $d_0\in(0,1)$. We recall that in this case the stencil consists of the three cells $\Omega_{j+l}, l=-1,0,1$, $P^{(2)}=\Popt \in \Poly{2}$ is the parabola interpolating in the sense of cell averages a given function $u(x)$ on the whole stencil, while $P^{(1)}_L$ and $P^{(1)}_R$ are the two left and right linear functions interpolating the cell averages $\ca{u}_{j-1}, \ca{u}_j$ and $\ca{u}_{j}, \ca{u}_{j+1}$, respectively.

\begin{example}\label{prop:CWENO3:disc}
Consider the operator $\CWENO(P^{(2)},P^{(1)}_L,P^{(1)}_R)$, with $d_L=d_R$ as defined by \cite{LPR:99}, and apply it to the cell averages of a Heaviside function and in particular to
\[
\ca{u}_{j-1}=1
\qquad
\ca{u}_{j}=0
\qquad
\ca{u}_{j+1}=0 
.
\]
By direct computation one finds that
\begin{equation} \label{eq:CW3disc:ratio}
\frac{I[P_0]}{I[P^{(2)}]}
=
\frac{3d_0^2-6d_0+16}
{16d_0^2}.
\end{equation}
Recalling that we are interested only in the domain $d_0\in(0,1]$,
since the derivative of \eqref{eq:CW3disc:ratio} vanishes at $d_0=16/3$, this expression attains its minimum on the boundary and precisely at $d_0=1$, where it attains the value $13/16$. Moreover, this ratio is clearly continuous provided $d_0\geq \delta>0$. Thus we have that for every choice of $0<\delta \leq d_0\leq1$, 
 $I[P_0]\asymp 1$ whenever $I[P^{(2)}] \asymp 1$.
\end{example}

We now turn to the general case, showing that Property R is verified by all one-dimensional $\CWENO$ reconstruction procedures with $d_0\neq0$.

\begin{theorem}
	Let $\CWENO(\Popt,P_1,\dots,P_{\emme})$ be a reconstruction with $\Popt\in\Poly{G}$ and $P_k$ of degree at most $g<G$ for all $k=1,\ldots,\emme$, with $d_0\geq\delta>0$. If a jump discontinuity is present in the stencil of the reconstruction polynomial, then  $I[P_0] \asymp 1$.
\end{theorem}

\begin{proof}
	Since $I[P_0]$ is bounded for $h\to0$ by definition \eqref{eq:JiangShuInd}, 
	in order to prove the statement we verify that $I[P_0]$ is larger than a quantity of order $h^0$. From the definition of the Jiang-Shu indicators \eqref{eq:JiangShuInd}, we notice that
	\[
		I[P_0] = \sum_{l=1}^{G} h^{2l-1} \int_{\Omega} \left( \frac{\mathrm{d}^l}{\mathrm{d}x^l} P_0 \right)^2 \mathrm{d}x > h^{2G-1} \int_{\Omega} \left( \frac{\mathrm{d}^G}{\mathrm{d}x^G} P_0 \right)^2 \mathrm{d}x
		.
	\]
	Using \eqref{eq:P0}, the $G$-th derivative of $P_0$ becomes
	\[
		\frac{\mathrm{d}^G}{\mathrm{d}x^G} P_0 
		= G! \frac{b_G}{d_0}
		=  \frac{(G+1)!}{d_0} \tilde{\delta}_{-g,G+1}
		,
	\]
	where  the leading coefficient $b_G$ of the optimal polynomial $\Popt$ has been computed as follows. Since $\Popt$ is an interpolant polynomial, using equation \eqref{eq:p:monomial} for $k=G$ and \eqref{eq:tildegamma}, we get
	\[
		b_G = \tilde{\delta}_{-g,G+1} \tilde{\Gamma}_{g,G+1,G+1}^G
		=  (G+1) \tilde{\delta}_{-g,G+1}
		.
	\]
	The $G$-th derivative of $P_0$ thus contains only the highest order divided difference of the optimal polynomial $\Popt$, which, in case of a discontinuity, diverges at a rate $h^{-G}$. In fact, one can find
	\[
		\tilde{\delta}_{-g,G+1} 
		= \frac{\sum_{i=0}^G (-1)^i \binom{G}{i} \ca{u}_{-g+i}}{(G+1)! h^G}
		\sim \frac{C}{h^G}
		,
	\]
	where $C\neq 0$ can depend on the size of the jump but not on $h$.
	We can finally compute
	\begin{align*}
		I[P_0]
		> 
		h^{2G-1} \int_{\Omega} \left( \frac{(G+1)!}{d_0} \tilde{\delta}_{-g,G+1} \right)^2 \mathrm{d}x
		\sim \left(\frac{(G+1)!}{d_0}\right)^2 C^2
		,
	\end{align*}
which concludes the proof.
\end{proof}

\subsection{Discontinuity in the reconstruction cell}
We now turn to point \ref{request:discCx} of Summary \ref{summary:WENO}. Let us consider the case in which the reconstruction is sought for the cell averages of a function with a discontinuity located inside the central cell. 
Clearly in this case all stencils of the polynomials involved in the reconstruction contain the troubled cell.

Consider first the cell averages of $u(x)=H(x)+v(x)$ where $v(x)$ a Lipschitz continuous function and $H(x)$ is an Heaviside function with jump located in the reconstruction cell.
First note that, thanks to Remark \ref{rem:IndLip},
which implies that 
$I[P]=\left.I[P]\right|_{v\equiv0}+\Ogrande(h)$,
the reconstructed values will differ at most by $\Ogrande(h)$ from those that one would obtain in the case $v\equiv0$.

Without loss of generality we now consider the case in which $\ca{u}_j=1$ for $j<0$, $\ca{u}_0=D\in(0,1)$ and $\ca{u}_j=0$ for $j>0$. 

We compute the $\CWENO$ reconstruction for $D\in(0,1)$, $d_0\in(0,1]$ at a generic point $x$ in the central cell. For $\CWENO3$ we choose the remaining coefficients symmetric, i.e. $d_L=d_R=(1-d_0)/2$, as in Proposition \ref{prop:CWENO3:disc}. For $\CWENO5$ we have one more parameter and we take $d_L=d_R=d_C/2$, i.e. $d_C=(1-d_0)/2, d_L=d_R=(1-d_0)/4$. 
For $\CWENO7$ we again give more weight to the central stencils taking $d_L=d_R=(1-d_0)/3$ and $d_{LL}=d_{RR}=(1-d_0)/6$, see also \eqref{eq:temporary:d}.

We are thus left with the free parameters $D$ and $d_0$ and applying the reconstruction we obtain a function $U(x;D,d_0)$. From these data, we fix $d_0$ and we extract $m_{d_0}(D)=\min_xU(x;D,d_0)$ and $M_{d_0}(D)=\max_xU(x;D,d_0)$. Figure \ref{fig:discCentraleCWENO35} 
shows the plots of $m_{d_0}(D)$ and $M_{d_0}(D)$ for all schemes and for several values of $d_0$ which are typical, namely $d_0=\nicefrac12$ (often employed in the literature), $d_0=\nicefrac34$ (used in the numerical experiments of this paper), and $d_0=\nicefrac{9}{10}$ (which overweights the central polynomial). It is clear that for all values considered, the reconstructed data are bounded by $[0,1]$ for all values of $D$ and thus no spurious oscillations are created and the total variation remains bounded.

It is noteworthy that the functions $m_{d_0}(D)$ and $M_{d_0}(D)$ depend so weakly on $d_0$. Moreover, we found comparable results for other choices of the coefficients in $\CWENO5$ and $\CWENO7$. Obviously, for $d_0$ very close to $0$ or $1$, $m_{d_0}(D)$ and $M_{d_0}(D)$ would change significantly. However, taking extreme values for $d_0$ does not make sense in practice: for $d_0\to0$, $P_0$ becomes undefined, while the limit $d_0\to1$ leads to $\Prec\to\Popt$ irrespectively of the oscillation indicators.

\begin{figure}
\begin{tabular}{ccc}
\includegraphics[width=0.32\textwidth]{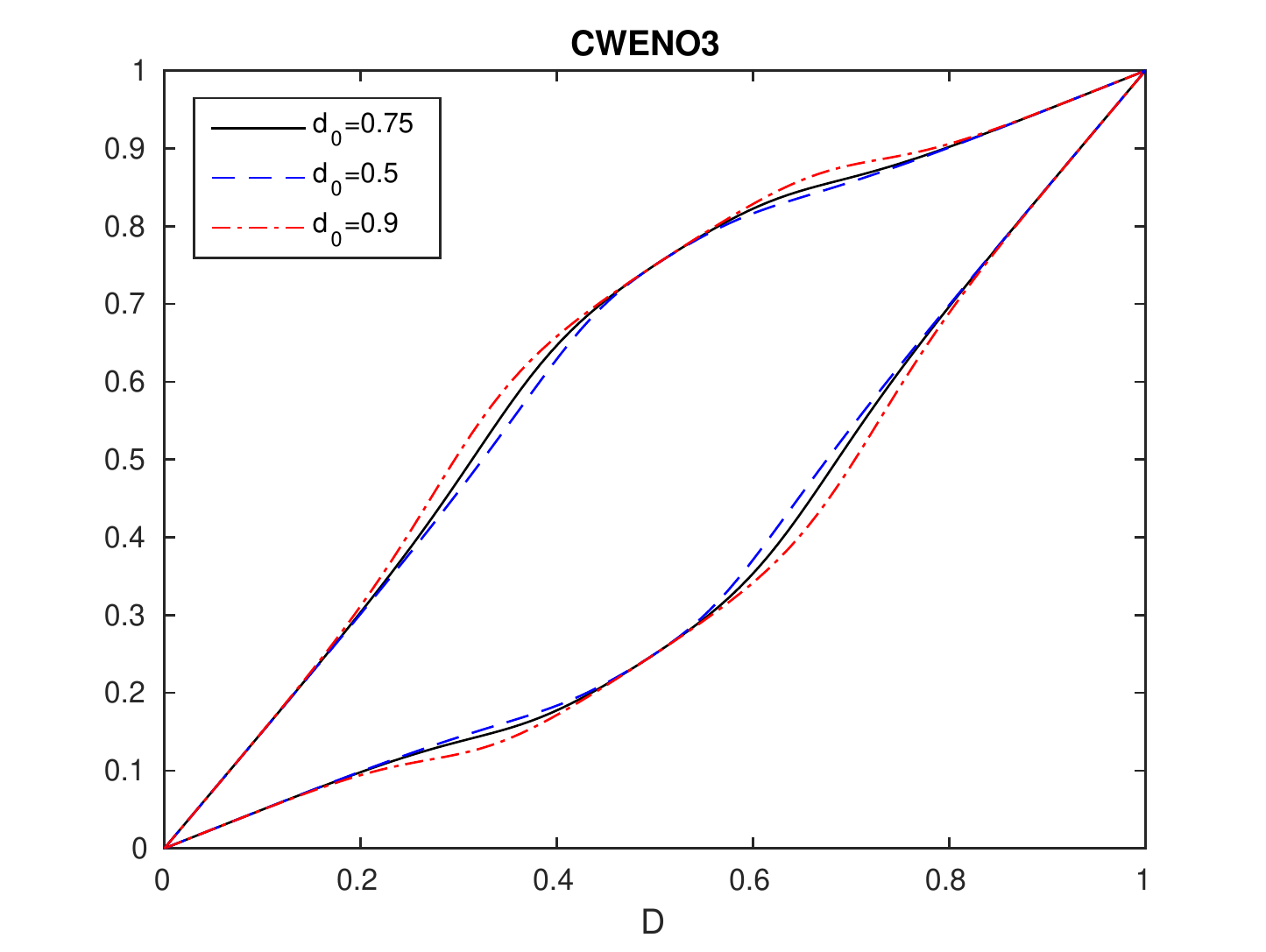}
&
\includegraphics[width=0.32\textwidth]{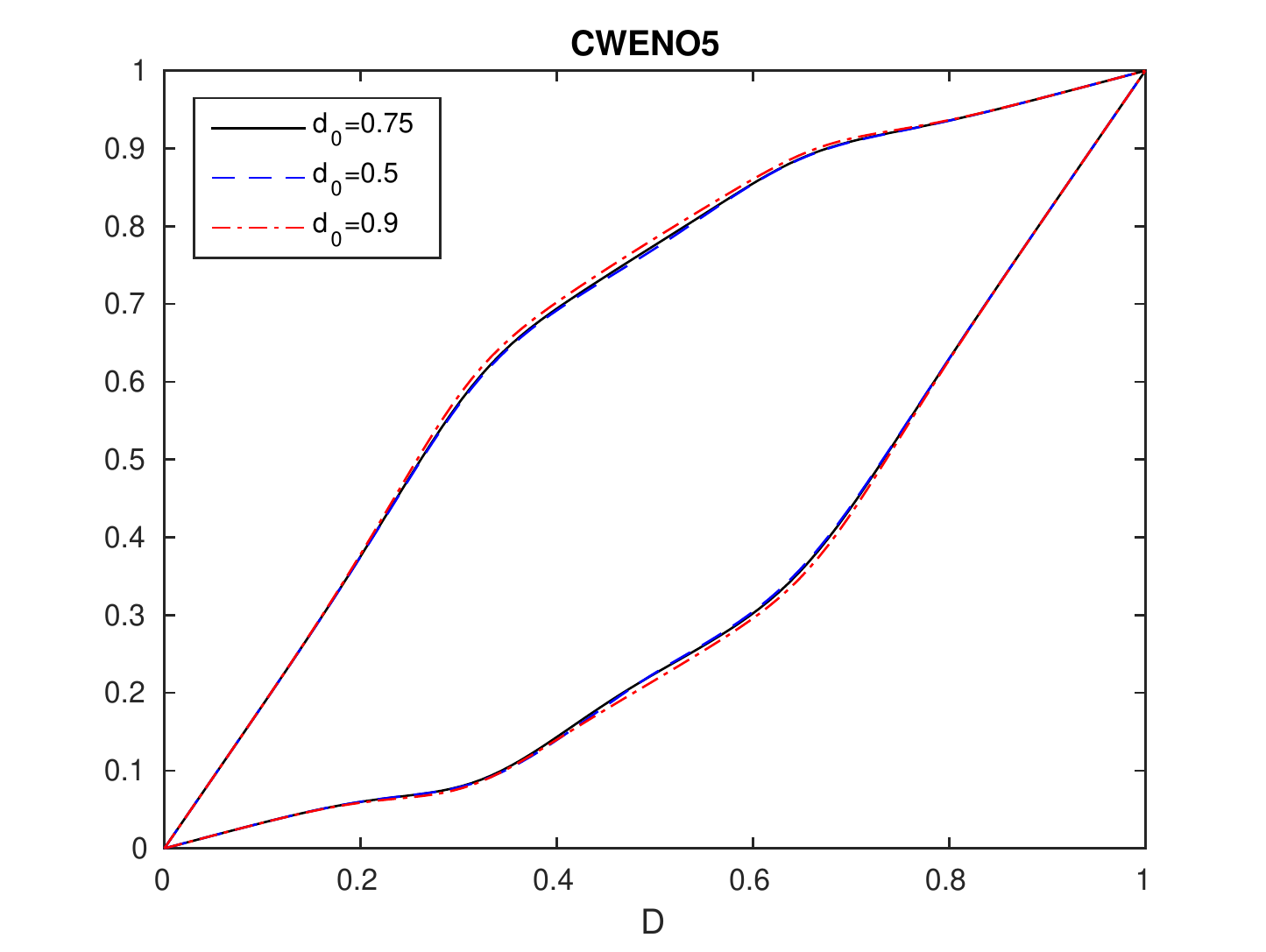}
&
\includegraphics[width=0.32\textwidth]{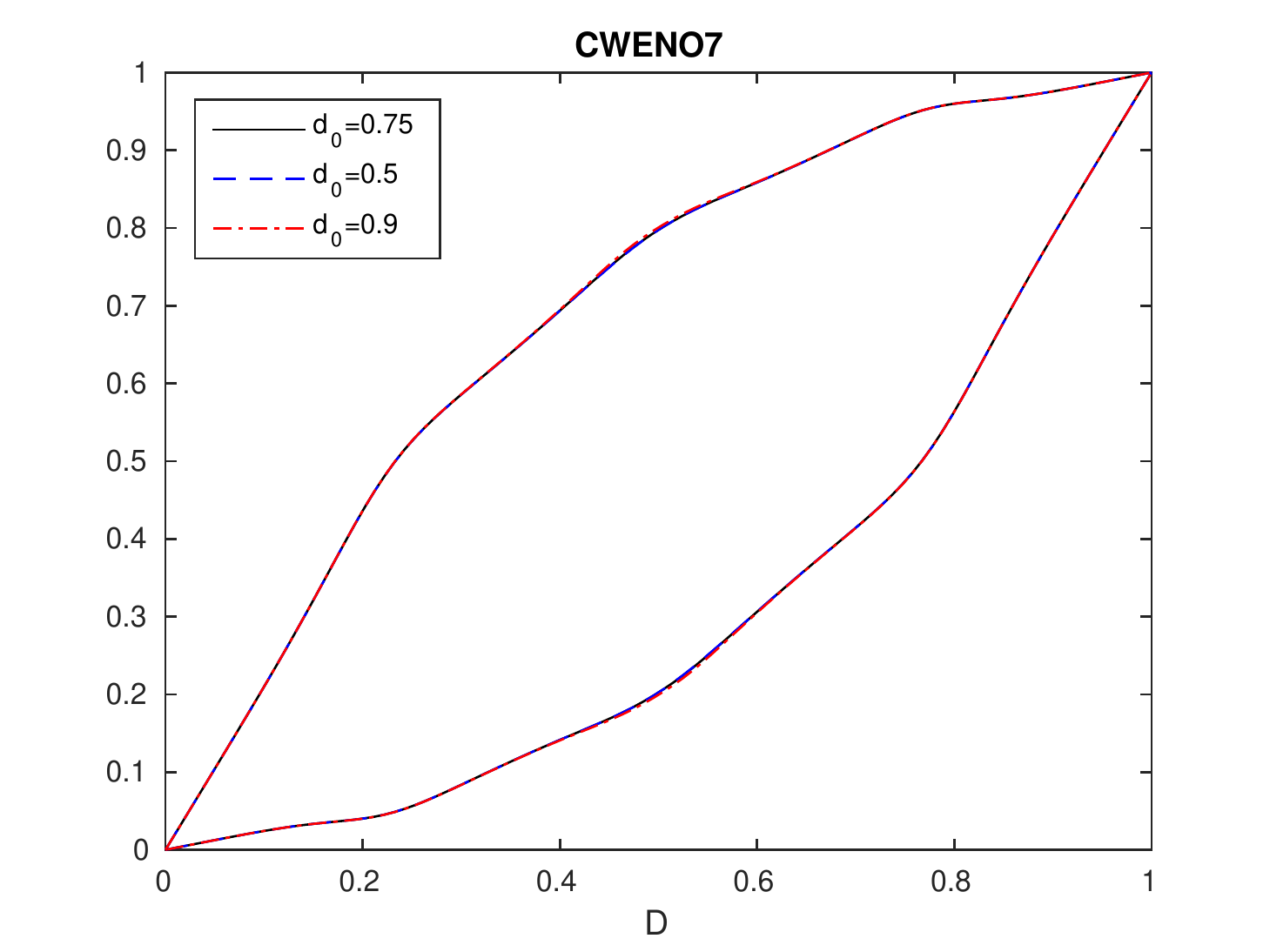}
\\
(a) & (b) & (c)
\end{tabular}
\caption{Discontinuity in the reconstruction cell. Minimum and maximum values attained by the reconstruction polynomial in the cell, as a function of the location $D$ of the discontinuity, for several values of $d_0$. Left: $\CWENO3$ with $d_L=d_R=(1-d_0)/2$. Middle: $\CWENO5$ with $d_C=(1-d_0)/2, d_L=d_R=(1-d_0)/4$. Right: $\CWENO7$ with $d_L=d_R=(1-d_0)/3, d_{LL}=d_{RR}=(1-d_0)/6$.}
\label{fig:discCentraleCWENO35}
\end{figure}

\section{Numerical experiments} \label{sec:numerical}

The purpose of the tests appearing in this section is to study the accuracy of the reconstructions proposed in this work, and to verify the non oscillatory properties of the resulting schemes. Thus we will consider the standard tests which are commonly used in the literature on high order methods for conservation laws: linear advection of smooth and non smooth waves, shock formation in Burgers' equation and Riemann problems for Euler gas dynamics. In all these cases, we will compare our results with solutions obtained with WENO schemes. Here, our results are comparable with standard WENO. 

Next, we will consider problems with sources, where our reconstructions are, we think, an improvement over standard WENO, because we easily evaluate the reconstructions at all quadrature points simultaneously. Again, we exhibit convergence histories and non oscillatory properties, using problems from shallow water and gas dynamics with source terms. Finally, we study the well balancing of the schemes built on the new reconstructions.

We construct numerical schemes applying the method of lines and the  Local Lax-Friedrichs flux with the $\CWENO3$, $\CWENO5$ and the newly proposed $\CWENO7$ and $\CWENO9$ reconstructions. The time integrators are Runge-Kutta schemes of matching order. In particular, the third order scheme employs the classical third order (strong stability preserving) SSP Runge-Kutta with three stages \cite{JiangShu:96}, the fifth order scheme the fifth order scheme with six stages of \cite[\S3.2.5]{Butcher:2008}, the scheme of order seven relies on the nine-stages scheme of \cite[pag 196]{Butcher:2008} and the scheme of order nine employs the scheme with eighteen stages of order ten of \cite{Curtis:1975}. Clearly, other Runge-Kutta or multistep schemes and different Riemann solvers could be used instead.

Source terms are integrated with a Gaussian quadrature formula  matching the order of the scheme when well-balancing is not an issue. In the case of the shallow water equations, we employ a scheme which is well-balanced for the lake at rest solution, constructed with the hydrostatic reconstruction technique of \cite{Audusse:2004}, the desingularization procedure proposed in \cite{Kur:desing} and the Richardson extrapolation for the quadrature of the source term. With reference to the latter, we employ the following quadratures $S^{(q)}$ of order $q$
\begin{align*}
S^{(4)} &= (4 S_2-S_1)/3\\
S^{(6)} &= (64S_4-20S_2+S_1)/45\\
S^{(8)} &= ( 4096 S_8-1344 S_4+84 S_2 -S_1 )/2835\\
S^{(10)} &=  1.450463049417298 S_{16}
			-0.481599059376837 S_8
			+ 0.031604938271605 S_4  \\
		& -0.000470311581423 S_2
			+0.000001383269357 S_1
,
\end{align*}
where $S_n$ denotes the quadrature of the source term computed with the composite trapezoidal rule with $n$ intervals on each cell. The first of these formulas was published in \cite{NatvigEtAl} and the other ones were derived by us following the ideas of that paper.

\subsection{Schemes for conservation laws}
In conservation laws, finite volume schemes on a fixed grid need reconstruction algorithms only to evaluate the numerical solution at the boundary of a cell. These data are used by the numerical fluxes to approximately solve local generalised Riemann Problems.

\begin{test}{Linear transport of smooth data, low frequency case.} \label{t:1} \end{test}

The convergence rates appearing in Fig. \ref{f:acc1} are obtained using an initial condition from\cite{Arandiga}. We solve $u_t+u_x=0$, on $[-1,1]$ with periodic boundary conditions, up to $T=2$, with initial condition
\[
u_0(x) = \sin\left( \pi x - \frac{1}{\pi} \sin(\pi x)\right).
\]

\begin{figure}\label{f:acc1}
\includegraphics[width=0.8\textwidth]{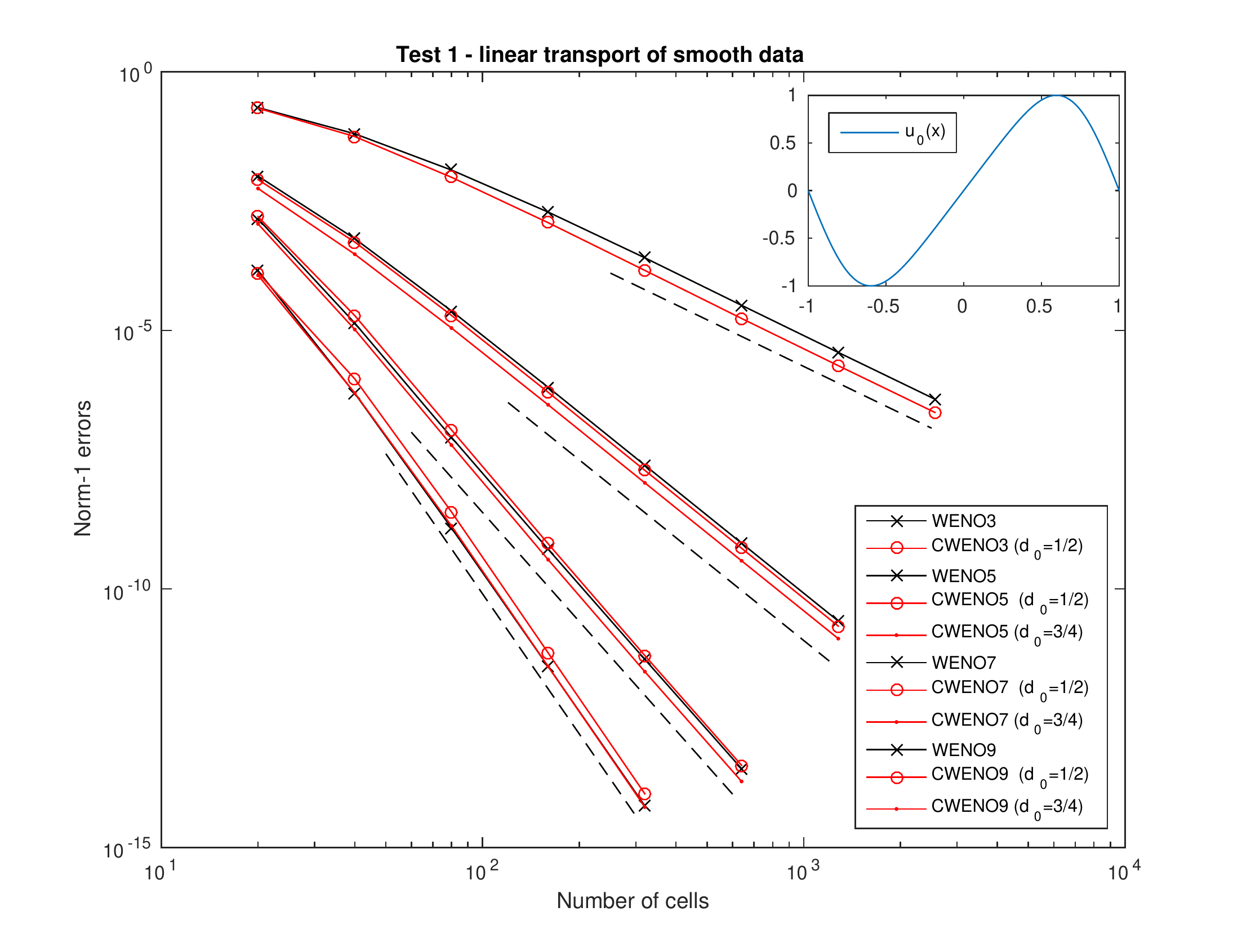}
\caption{Convergence rates for CWENO and WENO schemes of order 3, 5, 7 and 9, Test \ref{t:1}.} 
\end{figure}

The low order CWENO3 scheme has $d_0=\tfrac12$, while for the higher order schemes we show results with $d_0=\tfrac12$ (empty circles) and $d_0=\tfrac34$ (dots). Each group of curves is characterised with the desired slope (3, 5, 7 and 9 respectively, dashed black lines). The black solid curves are the reference results, obtained with the classical WENO scheme of the same order. Note that in all cases the errors almost coincide, with a very slight edge for the CWENO schemes with $d_0=\tfrac34$.

\begin{test}{Linear transport of smooth data, high frequency case.} \label{t:2} \end{test}

\begin{figure}\label{f:acc1}
\includegraphics[width=0.8\textwidth]{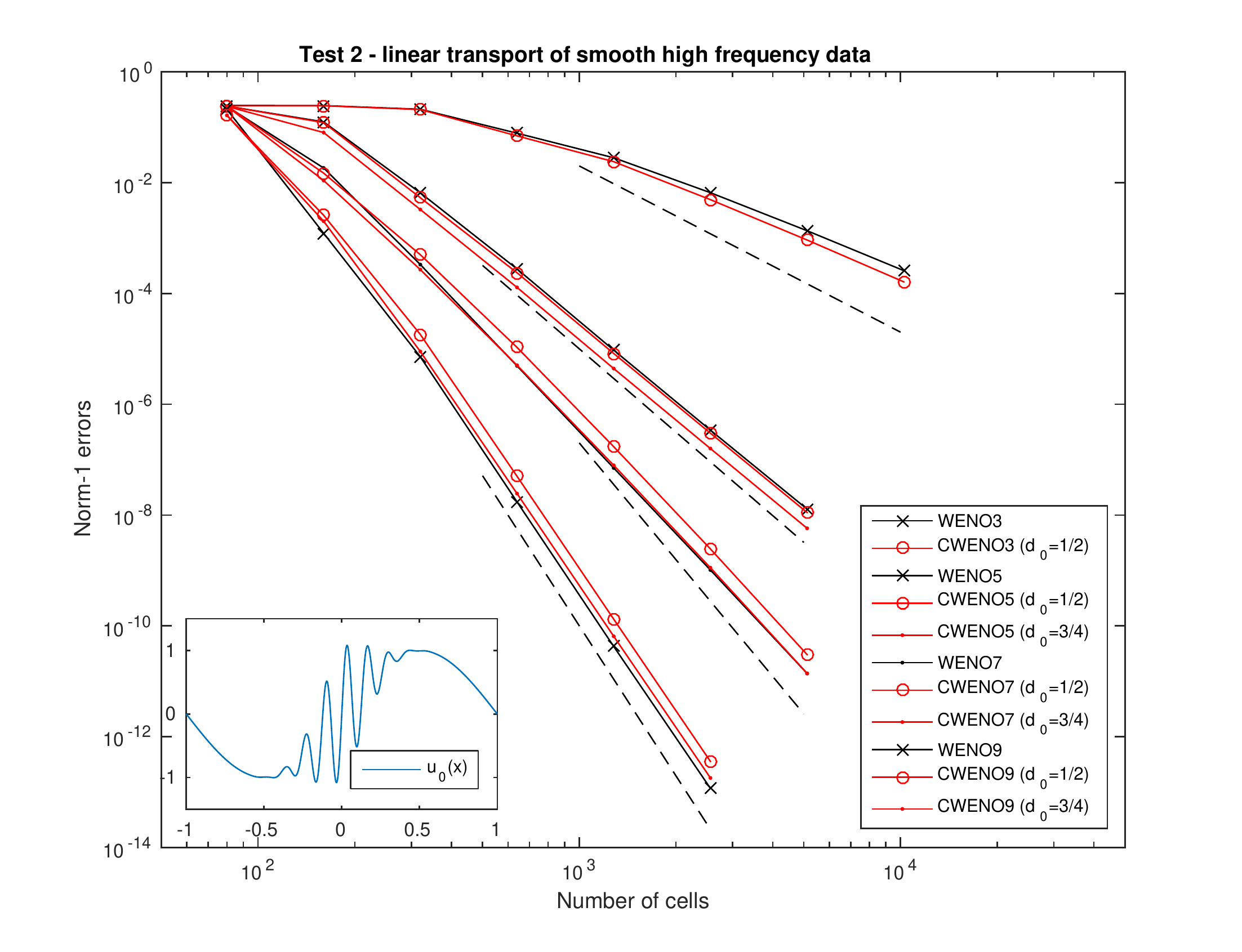}
\caption{Convergence rates for CWENO and WENO schemes of order 3, 5, 7 and 9, Test \ref{t:2}.} 
\end{figure}

This test is drawn from \cite{SCR:16}. It studies the propagation of a sine wave with a localised high frequency perturbation.
As before, we solve $u_t+u_x=0$, on $[-1,1]$ with periodic boundary conditions, up to $T=2$, but now the initial condition is
\[
u_0(x) = \sin\left( \pi x \right) +\tfrac14  \sin(15 \pi x)\;e^{-20x^2}.
\]
Again, the correct rates are achieved in all cases.
Note the high gain in accuracy obtained with the high order schemes even on coarse grids.

%
%
%
%
%
%

\begin{test}{Burgers' equation: shock interaction}\label{t:Burgers} \end{test}
This is a test on shock formation and shock interaction. We consider Burgers' equation in $[-1,1]$ with initial condition
\[
u_0(x) = 0.2 -\sin(\pi x) + \sin(2 \pi x) 
\]
and periodic boundary conditions.
The exact solution develops two shocks, which eventually collide, merging into a single discontinuity. We show three snapshots on the same panel in Figg. 
\ref{fig:burgersCWtq} and \ref{fig:burgersW}, with two zoom areas, which are enlarged on the right. The dashed black curve is the initial condition. The second curve is the solution at the time in which the two shocks develop ($T=1/(2 \pi)$). The third curve is slightly before shock interaction ($T=0.6$), with a detail enlarged in the figure appearing in the centre (zoom 1). The last curve is taken shortly after shock interaction ($T=1$), and a zoom of the interaction region is shown in the right panel (zoom 2). 


Fig. \ref{fig:burgersCWtq} shows the results obtained with CWENO schemes, with order 3, 5, 7, and 9 (black, blue, green and red curves respectively). The number of grid points is $N=160$. It is clear that the schemes do not produce spurious oscillations, and have an excellent resolution of discontinuities. As the order is increased, the profiles become sharper.
\begin{figure}
\centering
\includegraphics[width=0.3\linewidth]{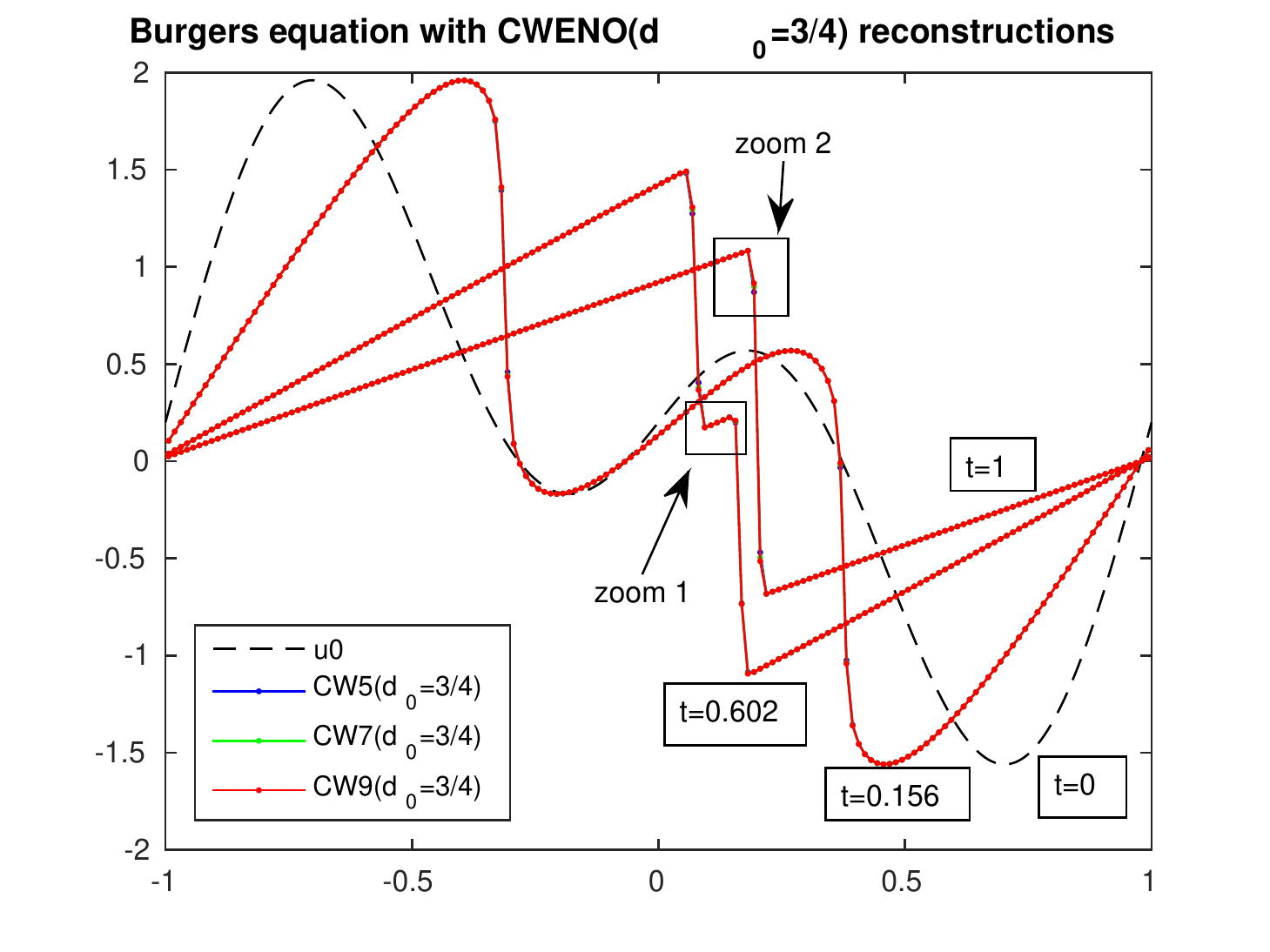}
\includegraphics[width=0.3\linewidth]{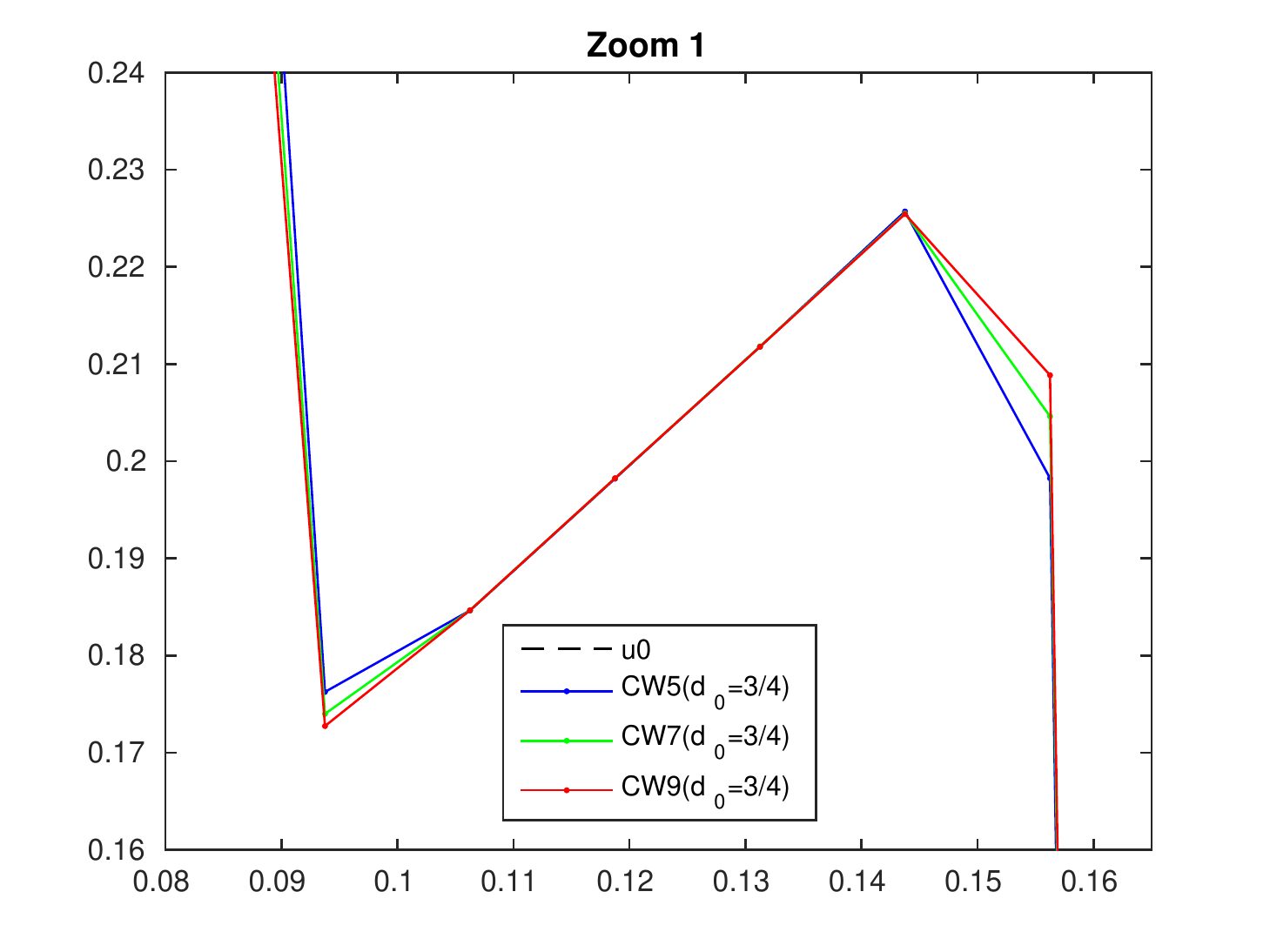}
\includegraphics[width=0.3\linewidth]{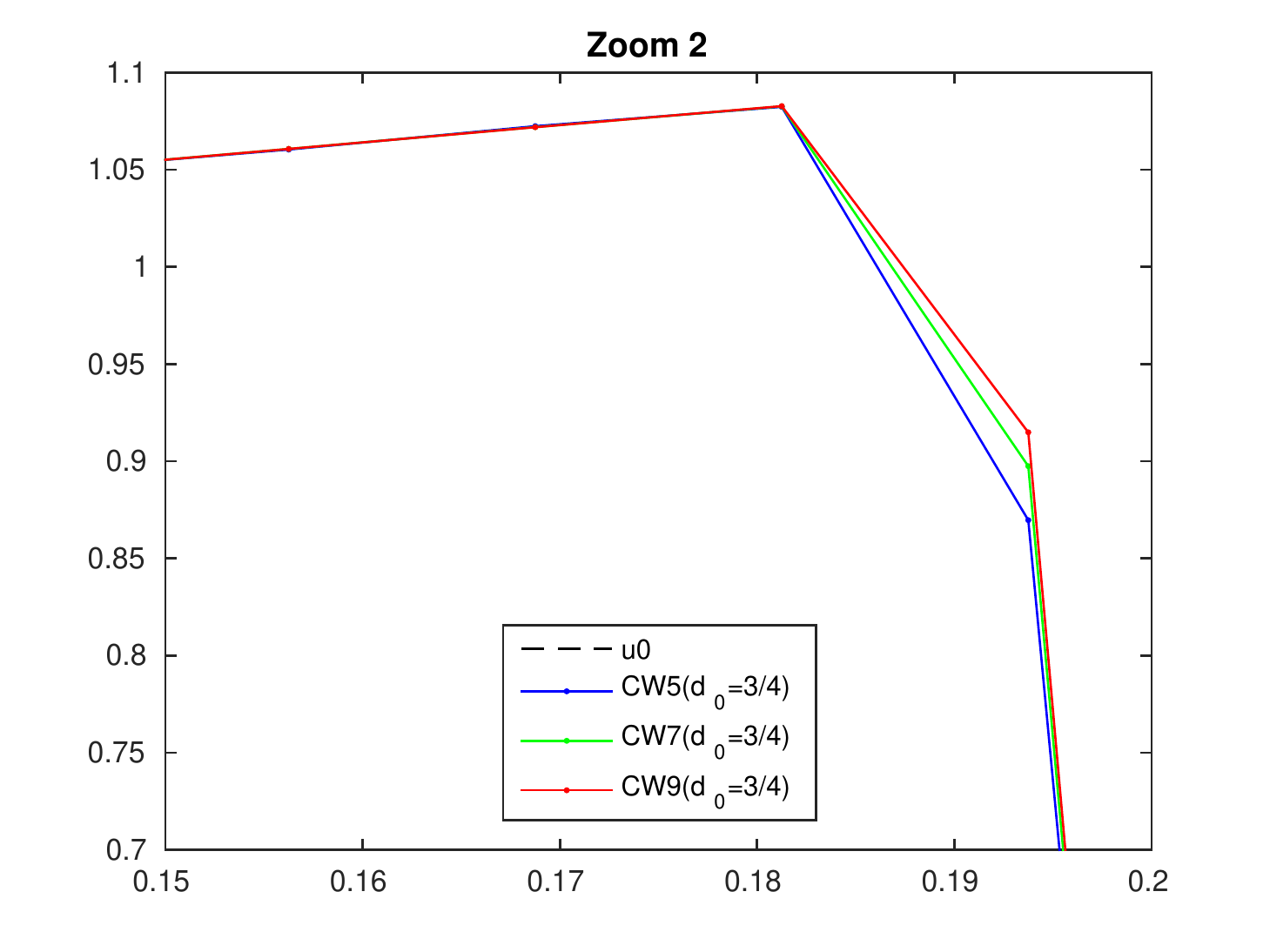}
\caption{Burgers' equation and shock interaction: CWENO schemes.
Evolution of the solution (left). Zoom slightly before (middle) and after (right) shock interaction.
}
\label{fig:burgersCWtq}
\end{figure}
For comparison, we also show the same results, obtained with the WENO scheme in Fig. \ref{fig:burgersW}. Note that the results are very similar.

\begin{figure}
\centering
\includegraphics[width=0.3\linewidth]{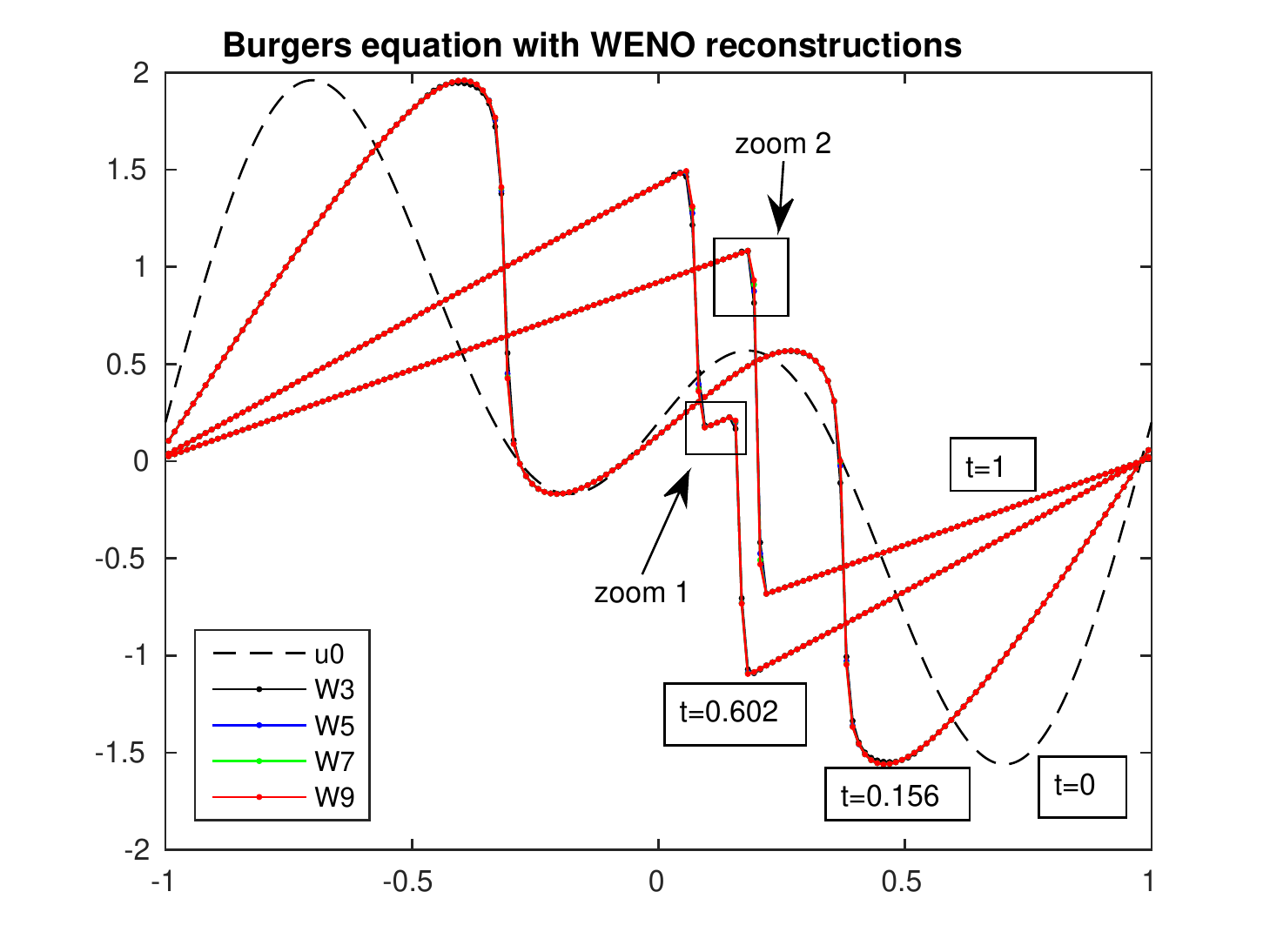}
\includegraphics[width=0.3\linewidth]{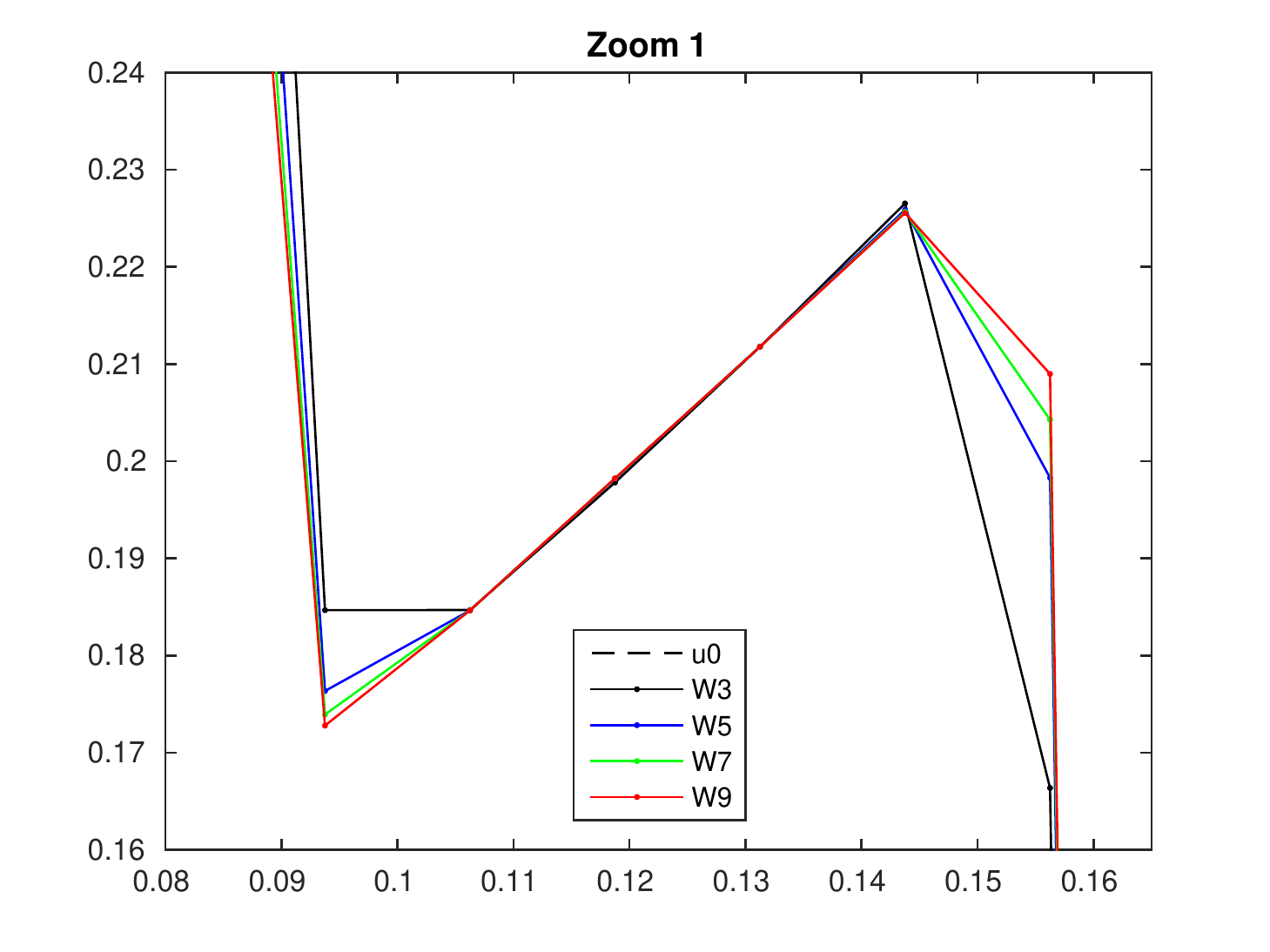}
\includegraphics[width=0.3\linewidth]{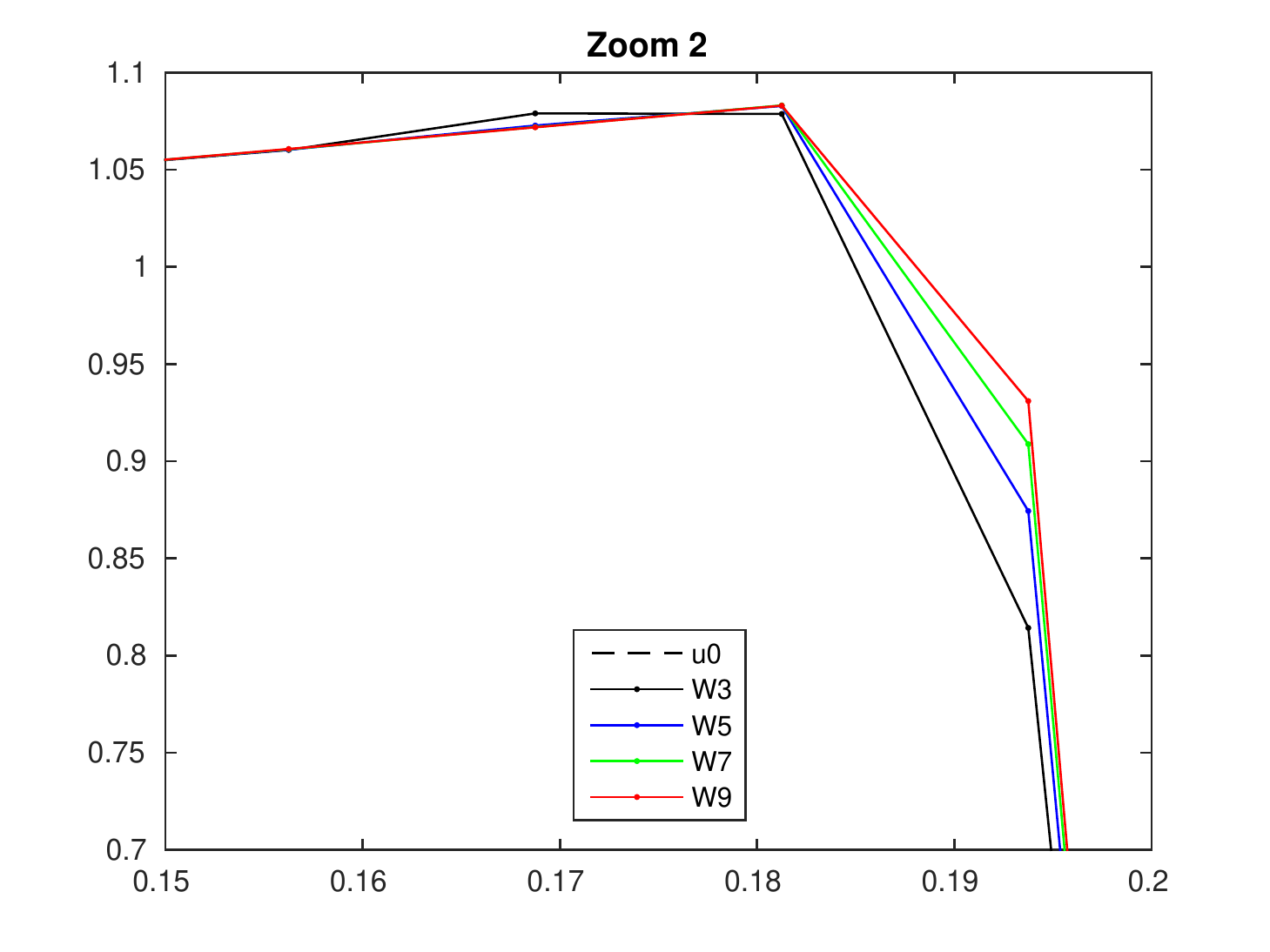}
\caption{Burgers' equation and shock interaction:  standard WENO schemes.
Evolution of the solution (left). Zoom slightly before (middle) and after (right) shock interaction.
}
\label{fig:burgersW}
\end{figure}

\begin{test}{Gas dynamics: Lax' Riemann problem}\label{t:Lax} \end{test}

The equations of gas dynamics for an ideal gas in one space dimension are
\[ \partial_t \left( \begin{array}{c}
\rho \\ \rho u \\ E
\end{array}\right) +
\partial_x \left( \begin{array}{c}
\rho u \\ \rho u^2 + p \\ u(E+p)
\end{array}\right)  = 0,
\]
where $\rho$ is the gas density, $u$ the velocity, $p$ the pressure, and $E$ the energy per unit volume. The pressure is linked to the other variables through the equation of state of an ideal gas, namely $p=(E-\tfrac12 \rho u^2)(\gamma-1)$, and we take $\gamma=1.4$. The Riemann problem by Lax has the following left and right states: $\rho_L=0.445, u_L=0.6989, p_L=3.5277$ and $\rho_R=0.5, u_R=0, p_R =0.571$. The solution develops a rarefaction wave travelling left, a contact discontinuity and a shock, both with positive speeds. The most interesting region is the density peak which occurs between the contact and the shock wave, where high order essentially non oscillatory schemes are known to develop spurious oscillations. For this reason, we show only a zoom on the density peak. It is well known that essentially non oscillatory and WENO schemes develop oscillations with amplitude decreasing under grid refinement, while their amplitude increases with the order of the scheme, at a given mesh width. 

The oscillations are originated by the interaction between waves in the first stages of the solution, when the discontinuities are so close that the algorithm cannot find a smooth stencil. Thus, they can be partly cured computing the reconstruction along characteristic fields, where the waves are approximately decoupled, \cite{QiuShu:02}. 

\begin{figure}
\centering
\begin{tabular}{cc}
\includegraphics[width=0.45\linewidth]{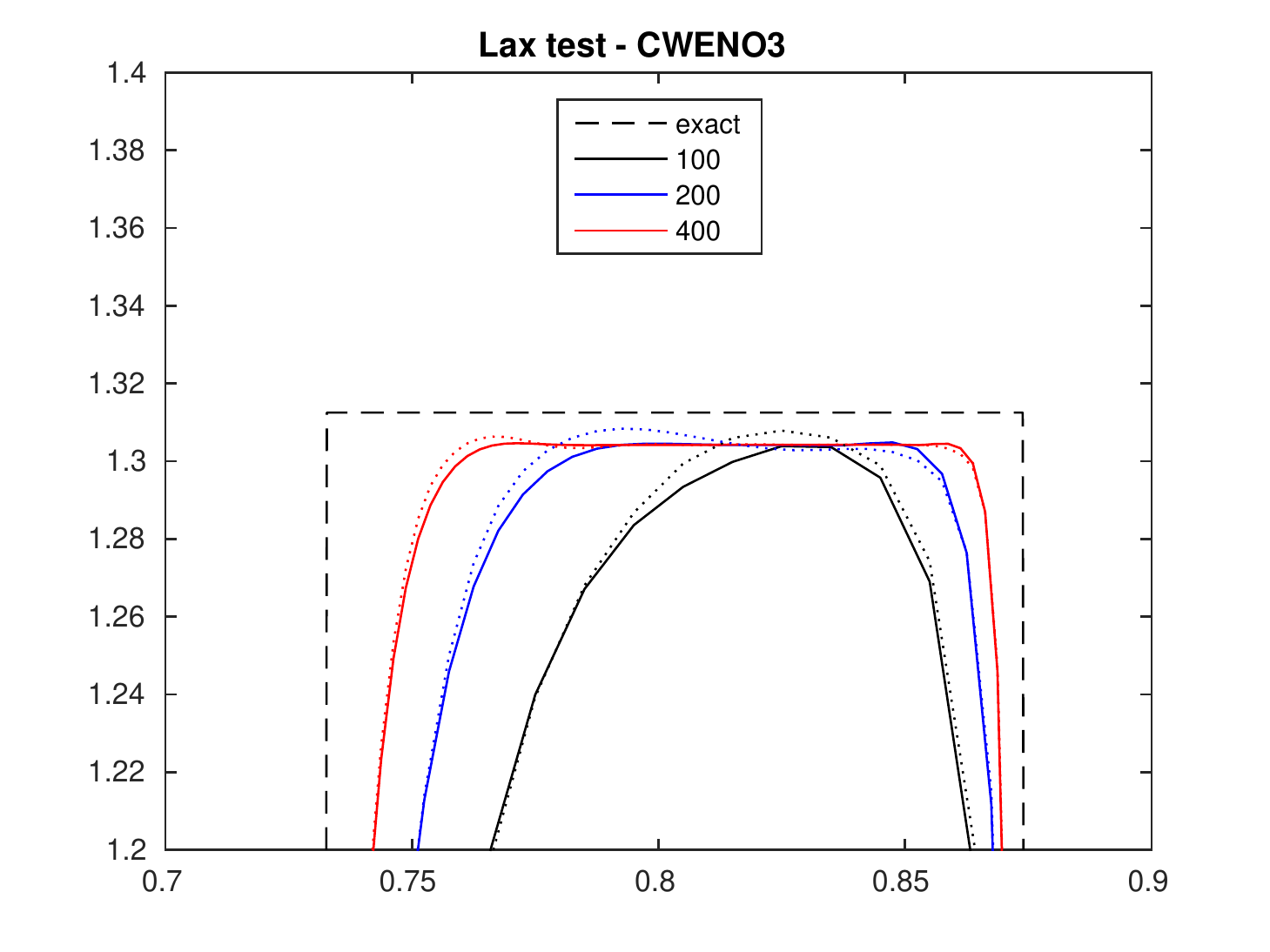}
&
\includegraphics[width=0.45\linewidth]{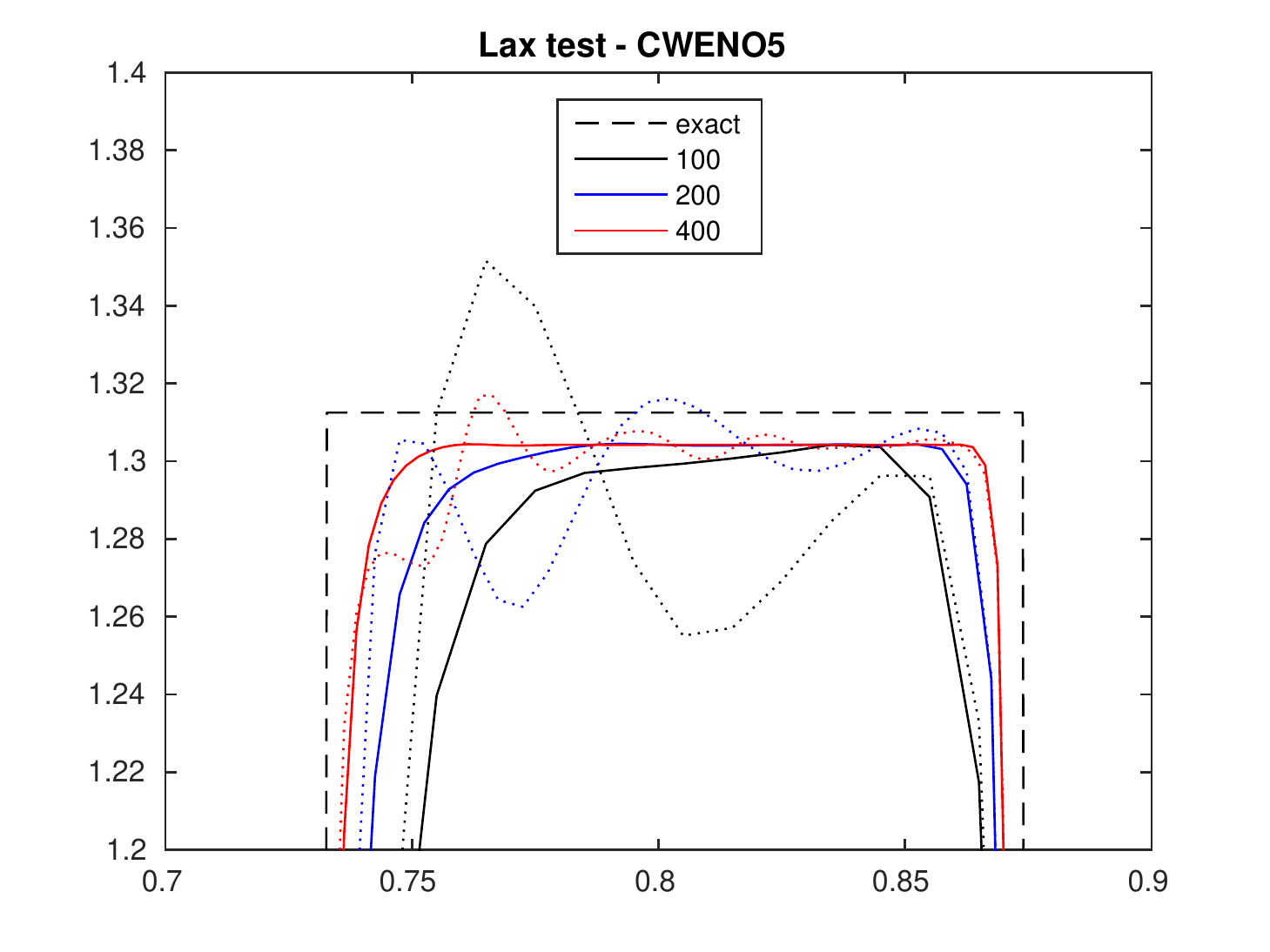}
\end{tabular}
\caption{Lax' test. Zoom on the density peak. $\CWENO3$ (left) and $\CWENO5$ (right) on several grids. The reconstruction is computed along characteristic directions (continuous lines) and on conservative variables (dotted lines).}
\label{fig:Lax_topzoom35}
\end{figure}

Fig. \ref{fig:Lax_topzoom35} contains the density peak obtained with CWENO3 (left) and CWENO5 (right) schemes. The continuous lines correspond to reconstructions computed along characteristic directions, for which the data in the whole stencil are projected along characteristic direction, before the reconstruction is computed, while the dashed curves are the standard reconstruction on conservative variables. Each figure contains the data obtained with $N=100, 200$ and $400$ grid points (black, blue and red curves, respectively). The improvement obtained with characteristic projection is quite dramatic, especially for the higher order schemes. In these two cases, the spurious oscillations disappear. Note also the improvement in the resolution of the waves with the high order CWENO5.

The following figure (Fig. \ref{fig:Lax_topzoom79}) contains the results obtained with CWENO7 and CWENO9 (top row). As a comparison, the same results with the standard $\WENO7$ and $\WENO9$ schemes are included in the bottom row plots of the same figure. As expected, the spurious oscillations become wilder for these high order schemes, unless the reconstruction is computed along characteristic directions.

\begin{figure}
\centering
\begin{tabular}{cc}
\includegraphics[width=0.45\linewidth]{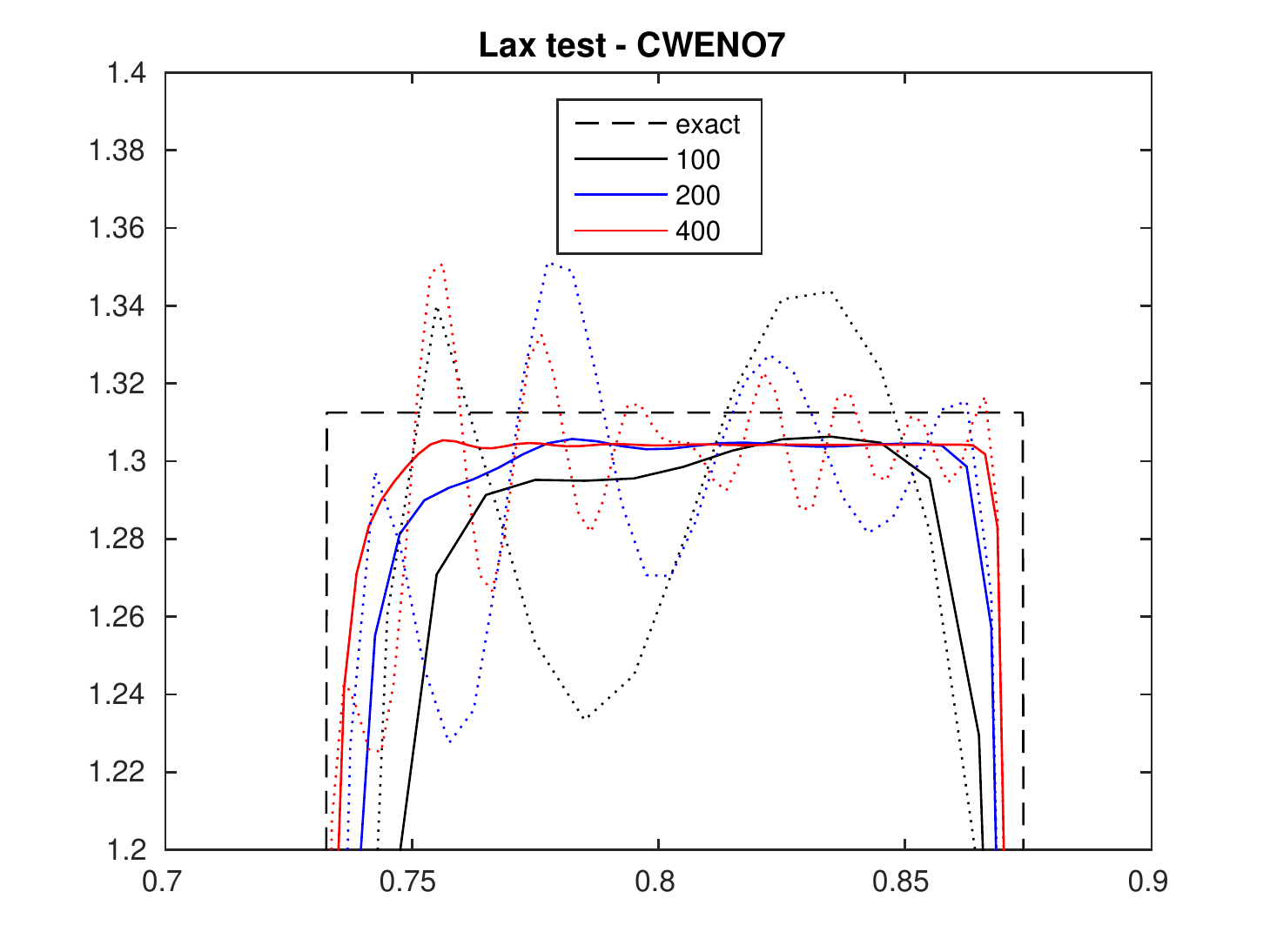}
&
\includegraphics[width=0.45\linewidth]{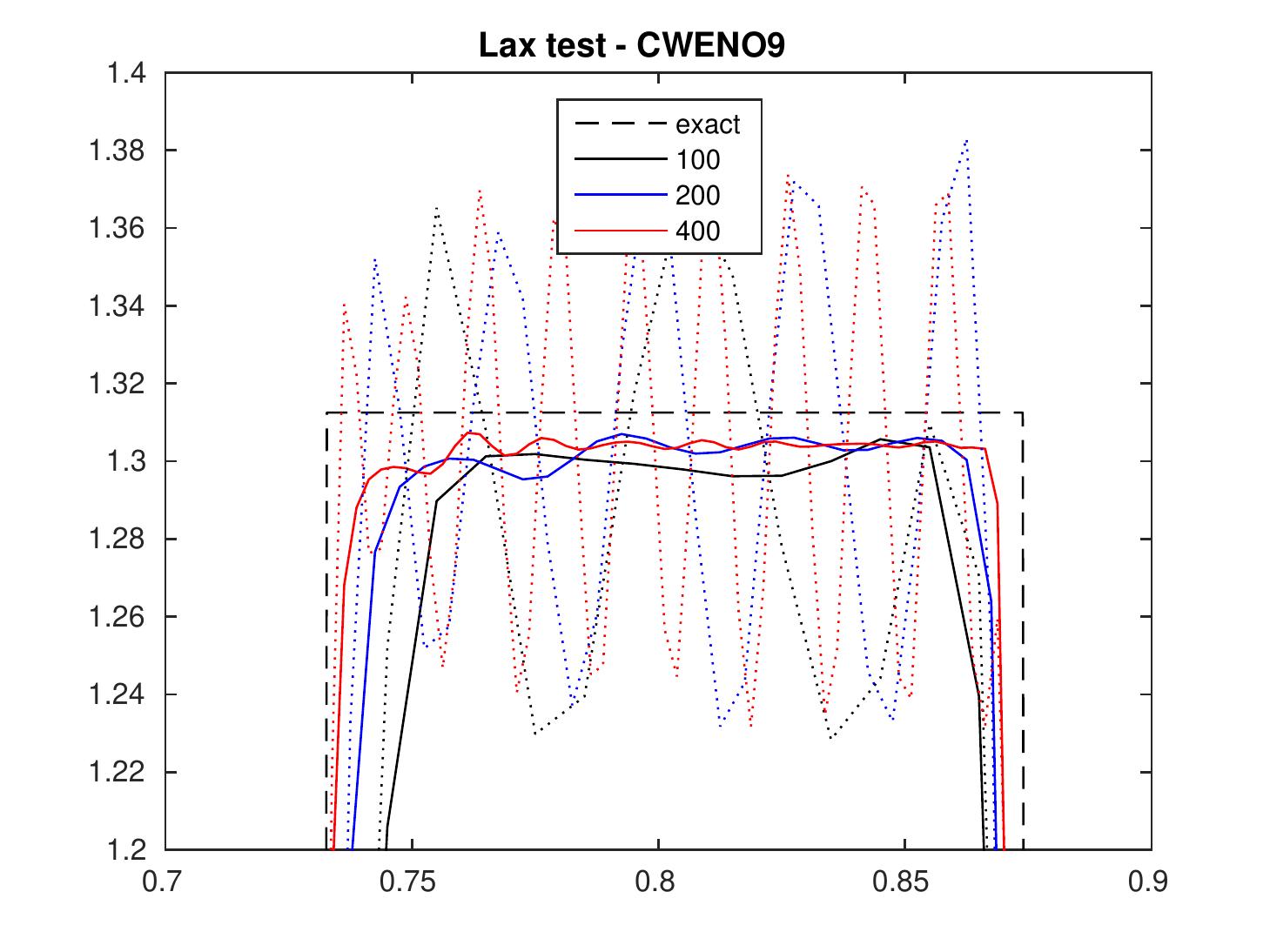}
\\
\includegraphics[width=0.45\linewidth]{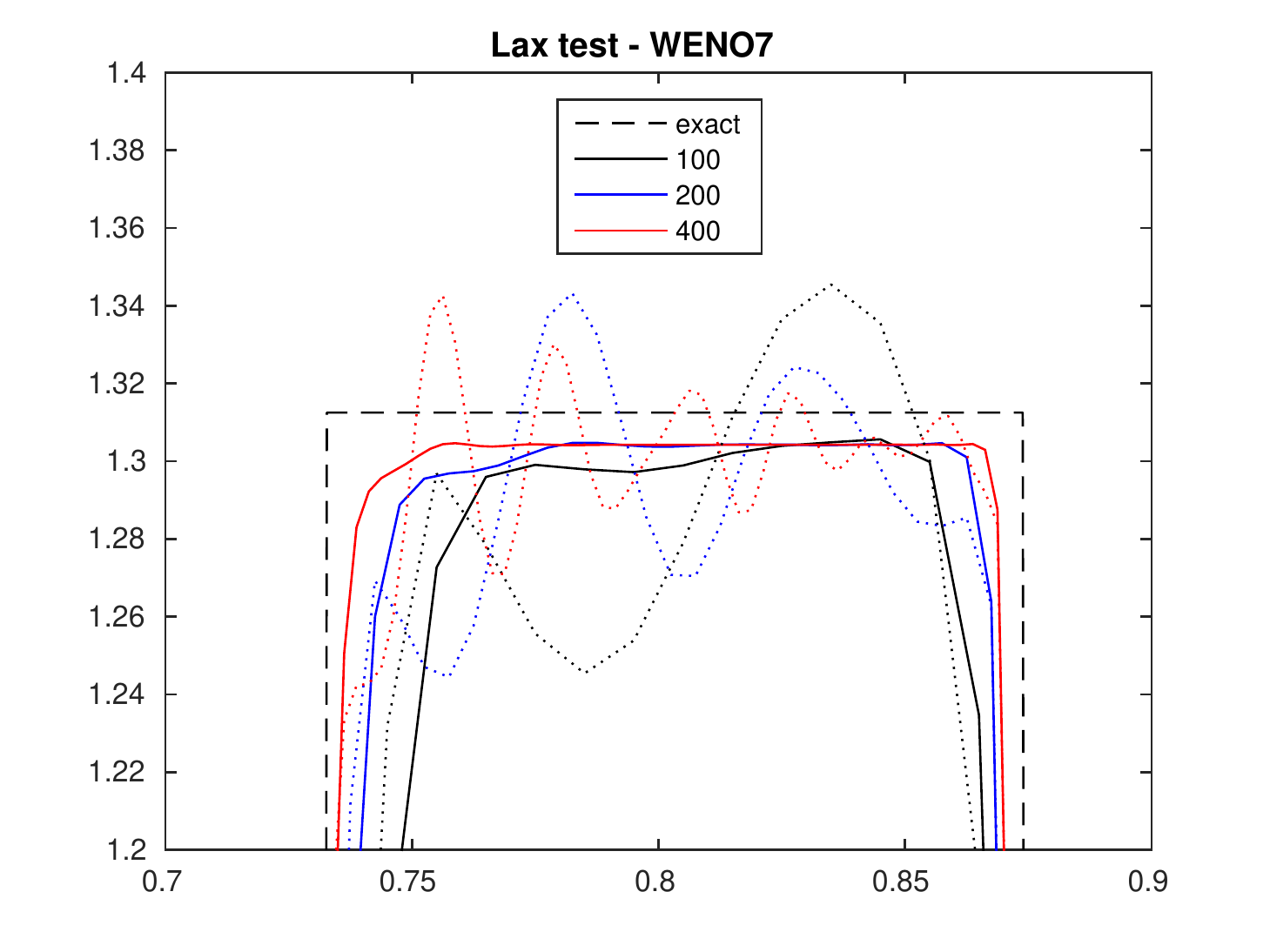}
&\includegraphics[width=0.45\linewidth]{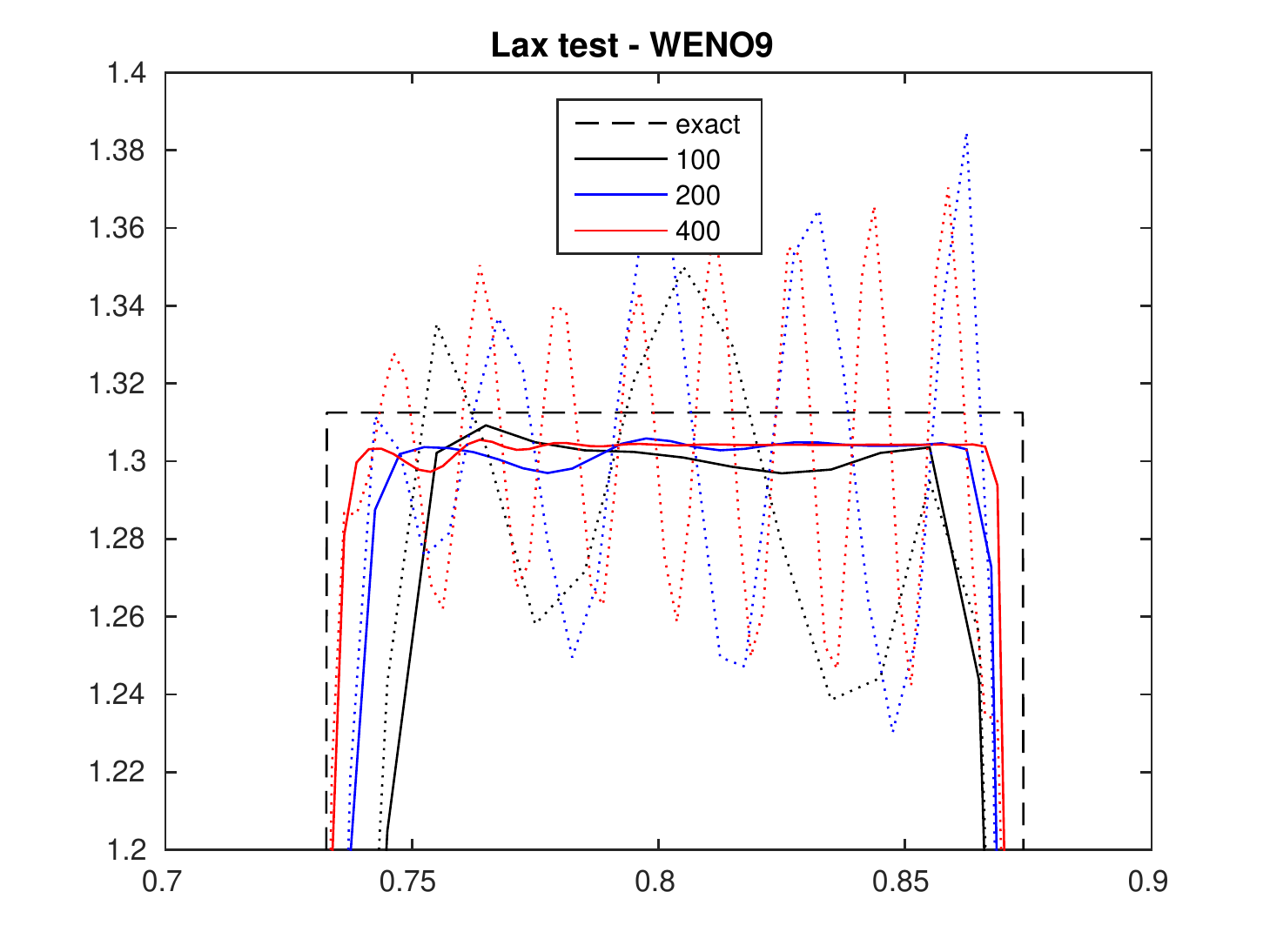}
\end{tabular}
\caption{Lax' test. Zoom on the density peak. $\CWENO7$ (top left) and $\CWENO9$ (top right), $\WENO7$ (bottom left) and $\WENO9$ (bottom right) on several grids. The reconstruction is computed along characteristic directions (continuous lines) and on conservative variables (dotted lines).}
\label{fig:Lax_topzoom79}
\end{figure}

The results discussed so far show that the new reconstructions are comparable to standard WENO reconstructions, not only as far as accuracy is concerned, but also in terms of non oscillatory, or essentially non oscillatory, properties. 
In both cases, for high order schemes, it is essential to employ characteristic projections, which could also be done in an adaptive way, as suggested in \cite{Puppo:adaptive} and\cite{PS:entropy}.

\subsection{Schemes for balance laws}
In balance laws, the reconstruction algorithm is used not only to evaluate the solution at the boundary of the cell, but also at interior nodes. In fact, the cell averages of the source term are evaluated with high order quadratures, which typically involve also interior nodes. Here, the $\CWENO$ technique permit to compute the cell averages of the source term with a single reconstruction.

\begin{test}{Shallow water equations: convergence rates on a non-flat riverbed}\end{test}
We consider the shallow water system, namely
\begin{equation}\label{eq:swe}
	u=\begin{pmatrix} h\\q\end{pmatrix}
	\qquad
    f(u) = \begin{pmatrix} q\\q^2/h+ \tfrac12 g h^2  \end{pmatrix} 
	\qquad
	g(u,x) = \begin{pmatrix} 0\\-ghz_x\end{pmatrix}.
  \end{equation}
Here $h$ denotes the water height, $q$ is the discharge and $z(x)$ the bottom topography, while $g$ is the gravitational constant. 

Following \cite{XingShu:2005:WBSWEfd}, we compute the flow with initial data given by
\begin{equation}
\label{eq:test:Shu}
z(x)=\sin^2(\pi x)
\qquad
h(0,x) = 5+e^{\cos(2\pi x)}
\quad
q(0,x) = \sin(\cos(2\pi x)),
\end{equation}
with periodic boundary conditions on the domain $[0,1]$. At time $t=0.1$ the solution is still smooth and we compare the numerical results with a reference solution computed with the fourth order scheme and $16384$ cells. The 1-norm of the errors appears in Table \ref{tab:SWE:Shu}. The well balanced quadrature is computed using Richardson's extrapolation, based on the trapezoidal rule. This means that the source term average is computed using the two boundary value reconstructions and additionally 3, 7 and 15 internal
reconstructions to achieve 5th, 7th and 9th order accuracy respectively. We emphasise that all these reconstructed data are computed from a single $\CWENO$ reconstruction polynomial, using the same weights for all coefficients. Note that the order of accuracy is perfectly met, until machine precision is reached.

This test would be extremely demanding on a standard WENO reconstruction, since the non linear weights must be changed for each quadrature node.

\begin{table}
\begin{centering}
\begin{tabular}{|r|rr|rr|rr|rr|} \hline
 & \multicolumn{2}{c|}{$\CWENO 3$} & \multicolumn{2}{c|}{$\CWENO 5$}
  & \multicolumn{2}{c|}{$\CWENO 7$} & \multicolumn{2}{c|}{$\CWENO 9$}
\\
N & error & rate & error & rate & error & rate & error & rate
\\
\hline
16 & 4.62e-02 &      & 5.53e-03 &      & 1.34e-03 &       & 6.92e-04 &      \\
32 & 1.04e-02 & 2.16 & 4.13e-04 & 3.74 & 7.39e-05 & 4.18 & 2.83e-05 & 4.61 \\
64 & 2.10e-03 & 2.30 & 1.75e-05 & 4.56 & 6.74e-07 & 6.78 & 1.23e-07 & 7.85 \\
128 & 3.14e-04 & 2.74 & 5.78e-07 & 4.92 & 5.02e-09 & 7.07 & 3.45e-10 & 8.48 \\
256 & 3.55e-05 & 3.15 & 1.82e-08 & 4.99 & 3.91e-11 & 7.00 & 7.44e-13 & 8.86 \\
512 & 2.42e-06 & 3.88 & 5.71e-10 & 4.99 & 3.08e-13 & 6.99 &    & \\ \hline
\end{tabular}
\caption{Errors and convergence rates for SW convergence on a non flat riverbed.}
\label{tab:SWE:Shu}
\end{centering}
\end{table}


\begin{test}{Shallow water equations: well-balancing test on a rough bottom}\end{test}

This is a classical test, to explore the well balancing properties of a scheme, see \cite{NatvigEtAl}. We consider a flat lake $z(x)+h(x)\equiv 1.5$, with water at rest. The bottom cell averages are randomly extracted from a uniform distribution on $[0,1]$. Thus the function $z(x)$ is extremely irregular, but nonetheless the exact solution preserves the flat surface, and the water should remain still. A well balanced scheme preserves this solution at machine precision.

\begin{table}
\begin{centering}
\begin{tabular}{|l|cccc|cccc|}\hline
method 
& \multicolumn{4}{c|}{error in $q$}\\ \hline
& N=100 & N=200 & N=400 & N=800\\
$\CWENO 9$ 
   & 7.4471e-16 & 1.4354e-15 & 1.8279e-15 & 2.5115e-15 \\
$\CWENO 7$ 
   & 2.1206e-15  & 3.0564e-15 &  7.1562e-15 &  1.6473e-14\\
$\CWENO 5$ 
   & 1.7490e-15  &  3.0874e-15   & 5.3284e-15 &   9.9496e-15\\
$\CWENO 3$ 
   & 1.9032e-15   & 3.5655e-15 &   4.7854e-15 &   7.6668e-15 \\ \hline
\end{tabular}
\caption{Well balancing errors on a rough lake at rest. Discharge}
\label{tab:SWE:Rough}
\end{centering}
\end{table}
We report in Table \ref{tab:SWE:Rough} the values of the discharge computed by all CWENO schemes tested in this work for several grids. It is clear that in all cases the discharge is zero within machine precision, so that the  quadrature of the source is indeed well balanced in all cases, notwithstanding the fact that, again, it is computed with a single polynomial for all quadrature nodes.

The data on the water height have the same precision, and are not reported for brevity.

\begin{test}{Shallow water equation: dam-break over a hump}\end{test}

This test studies the movement of a shock and a rarefaction on a shallow water problem, with non constant bottom topography. The initial conditions for the water surface $H(x)=h(x)+z(x)$ and the discharge are
\begin{displaymath}
H(x,t=0) = \begin{cases} 1.5 & x<0 \\ 0.5 & x>0, \end{cases}
\text{ and } \quad
q(x,t=0) \equiv 0,
\end{displaymath}
on $[-2,2]$, and the bottom topography is $z(x)=0.3\,e^{-10x^2}$. The final time is $t=0.2$.
This set up contains a discontinuity on the amount of water, in correspondence with a hump in the bottom topography. As the solution develops, a shock moves towards the right, while a rarefaction wave travels left.
\begin{figure}
\centering
\includegraphics[width=0.45\linewidth]{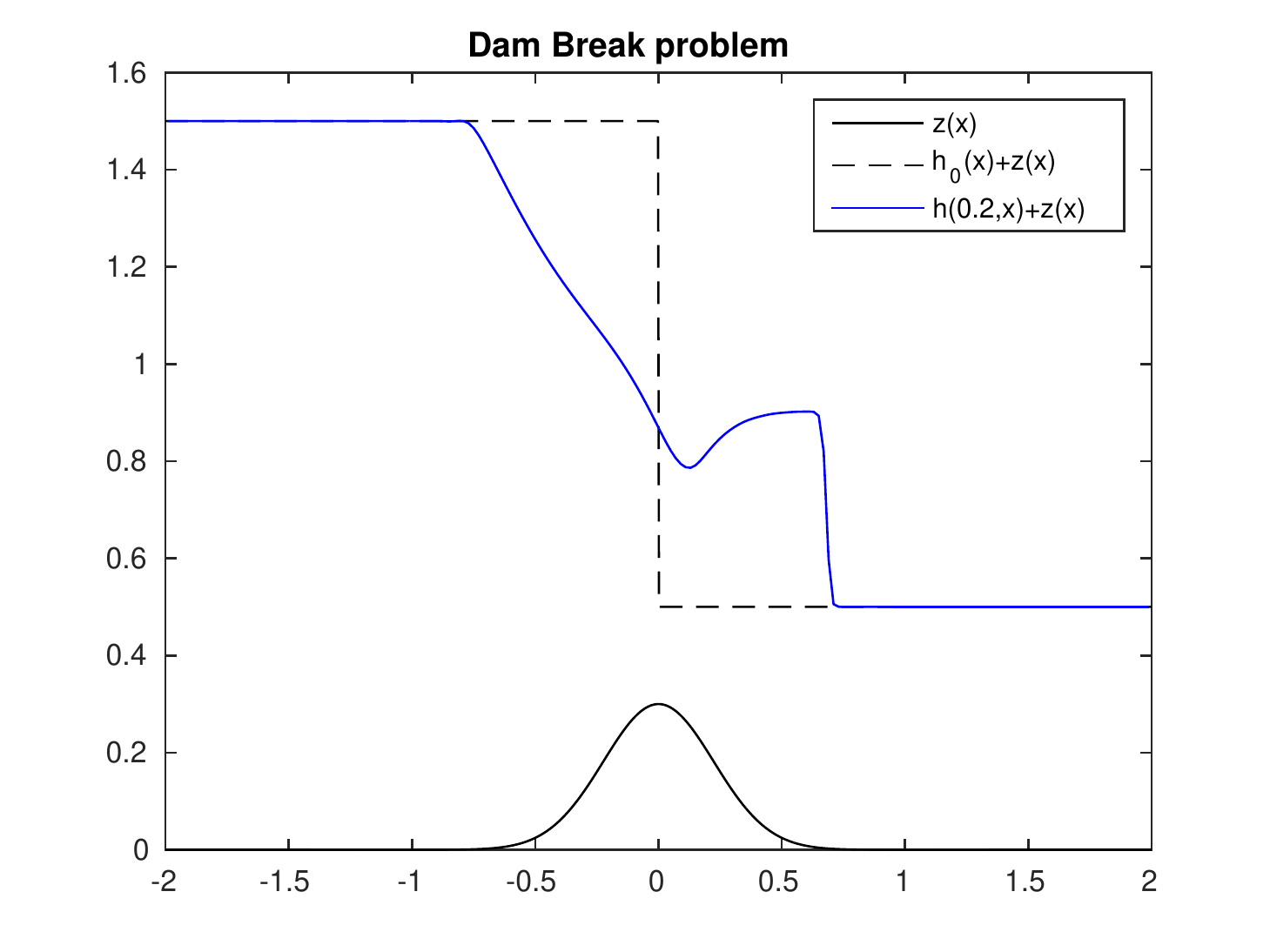}
\hfill
\includegraphics[width=0.45\linewidth]{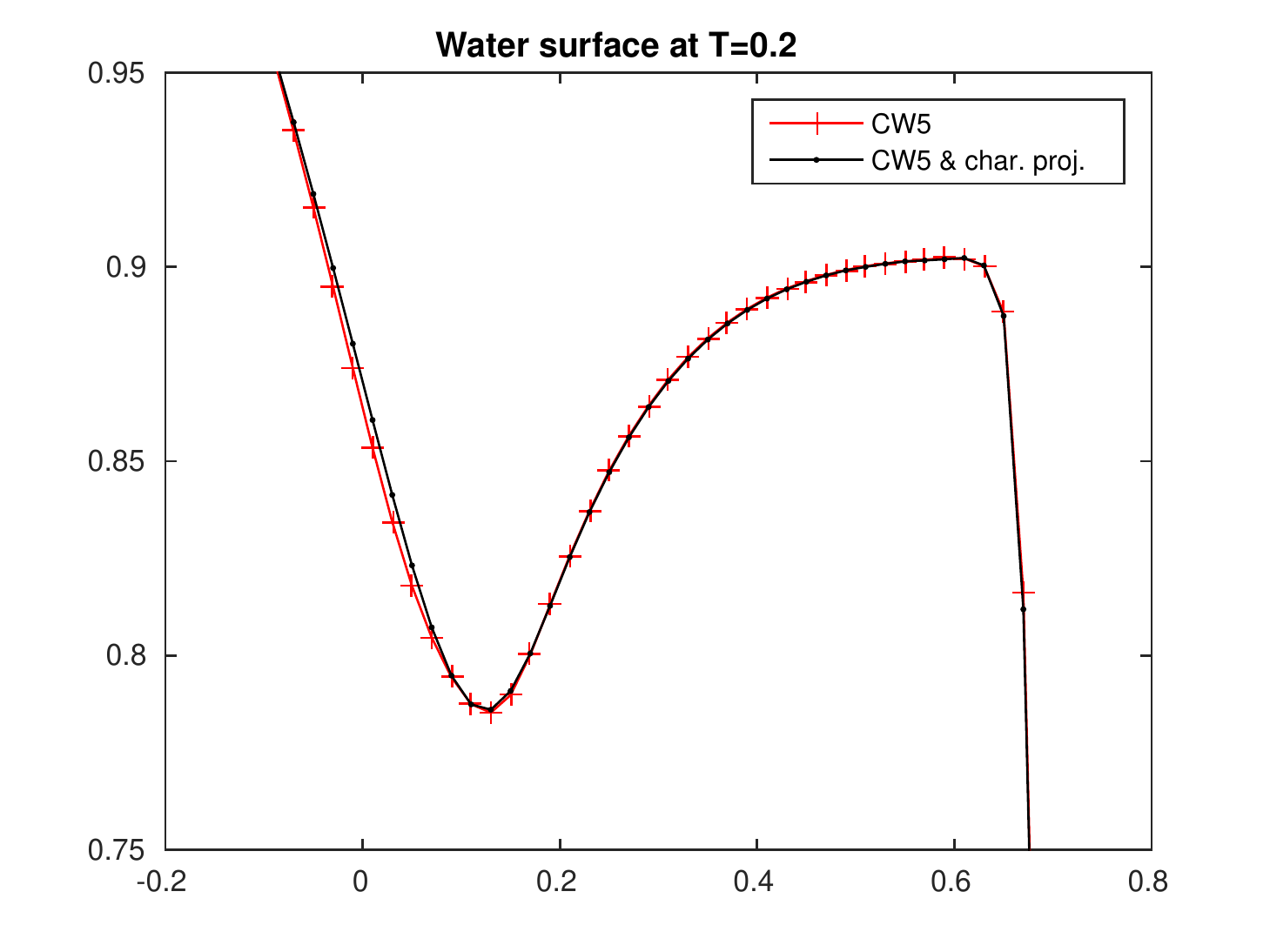}
\\
\includegraphics[width=0.45\linewidth]{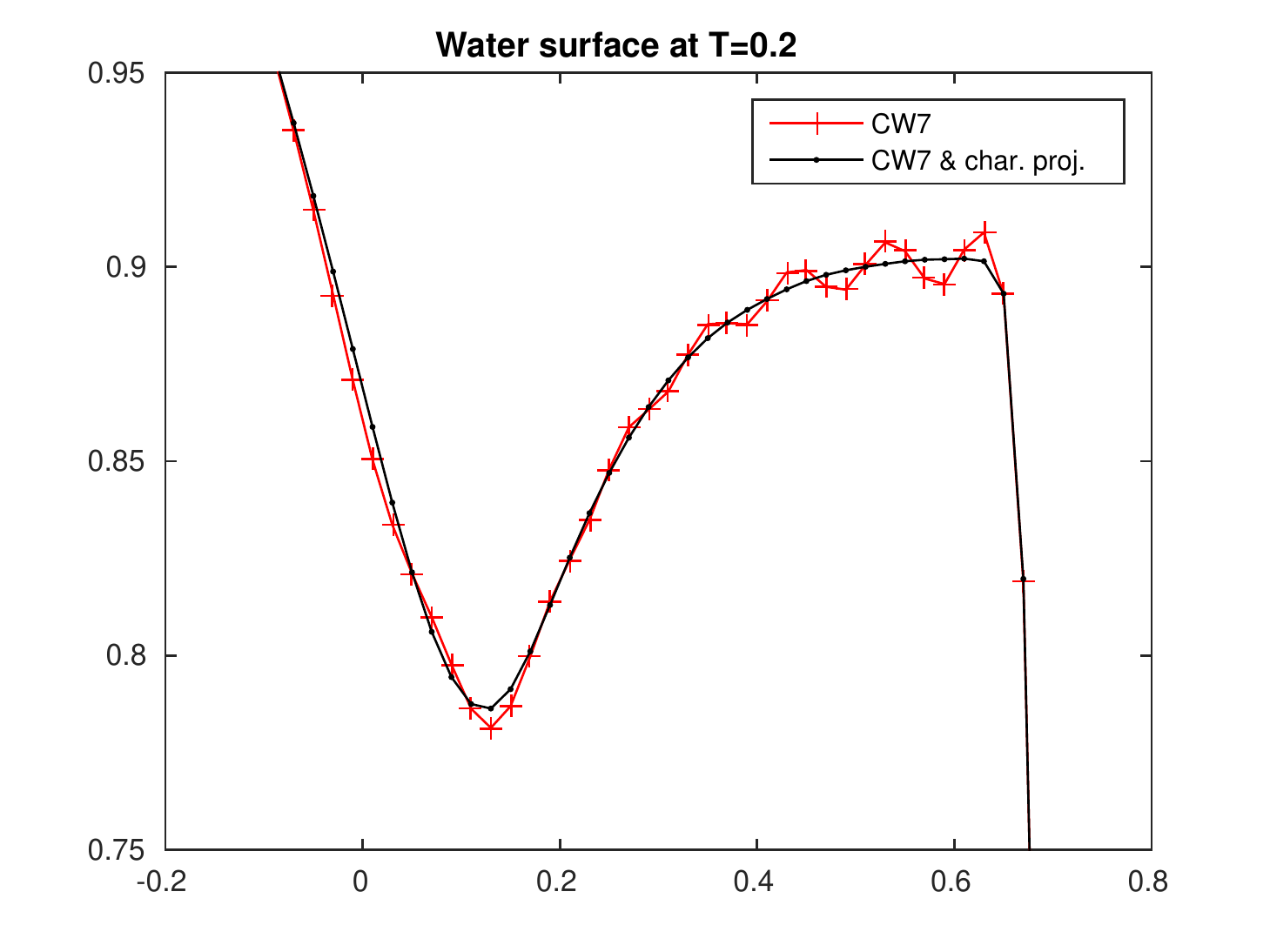}
\hfill
\includegraphics[width=0.45\linewidth]{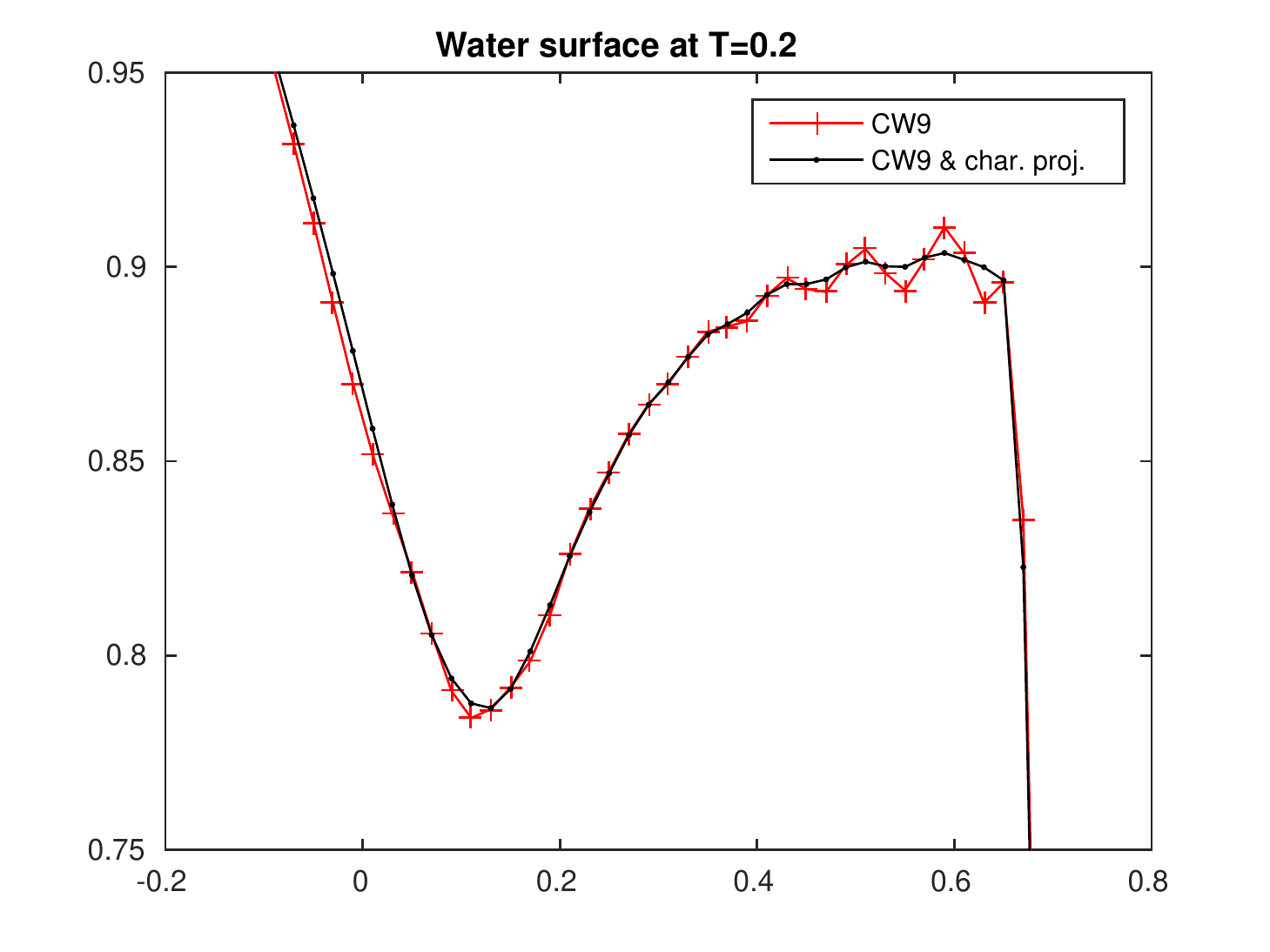}
\caption{Dam break over a hump. Top left: water height at time $t=0.2$. The remaining plots are zooms on the tail of the rarefaction and the jump, for $\CWENO 5, 7$ and 9. The black solid line is with characteristic projections.}
\label{fig:SWEdamBreak}
\end{figure}

The results on the water surface for $N=200$ are shown in Fig. \ref{fig:SWEdamBreak}, with zooms on the most difficult parts of the solution for the $\CWENO 5, 7$ and 9 schemes. Again, the numerical solution exhibits spurious oscillations behind the shock (red curve, with + markers), which can be levelled out using the characteristic projection, before evaluating the reconstruction (black solid lines). The same behaviour can be observed in the solution for the discharge.

\begin{test}{Gas dynamics: Riemann Problem in spherical coordinates}\end{test}
In the case of radial symmetry, the gas dynamics equations can be written as a 1D system, with a source term, which takes into account the geometrical effect, \cite[\S1.6.3]{Toro:book}.
Radially symmetric solutions of the Euler equations in $\R^n$ may be computed by solving

\[ \partial_t \left( \begin{array}{c}
\rho \\ \rho u \\ E
\end{array}\right) +
\partial_r \left( \begin{array}{c}
\rho u \\ \rho u^2 + p \\ u(E+p)
\end{array}\right)  
= 
-\frac{n-1}{r} 
\left( \begin{array}{c}
\rho u \\ \rho u^2 \\ up
\end{array}\right).
\]

We compute the so-called ``explosion problem'', which has a shock tube like initial data. In our case, we take Sod's test data, namely $(\rho_L,u_L,p_L) = (1,0,1)$ for $r<0.5$ and 
$(\rho_R,u_R,p_R) = (0.125,0,0.1)$ for $r>0.5$.
The final time of the simulation is $t=0.25$.

In order to avoid difficulties with the boundary conditions in the singular point $r=0$, and taking into account that the 
computed solution will have null velocity $u$ (and thus null source term) close to $r=0$, because of the initial data,
we computed the solutions for $r\in[-1,1]$ with symmetric initial data and free-flow boundary conditions. Gaussian quadrature formulas of appropriate order are employed to compute the cell average of the source term and the grid is chosen in order to avoid quadrature nodes at the singular point $x=0$.
The solution at final time obtained with $N=400$ cells is shown in the picture \ref{fig:radialSod}, restricted to the domain $r\in[0,1]$. Again, we show the density profiles, since the density contains the main features of the flow. The zoom in the density profile centred on the contact wave is shown for the reconstruction computed along conservative variables (central plot of the figure), and along characteristic variables (right plot).  Each plot contains the solution obtained with all four different schemes tested in this work. The cyan curve is given by $\CWENO 3$, and the improvement in the resolution of the contact wave obtained increasing the accuracy of the scheme is quite apparent. Here too, only one reconstruction polynomial is needed for each Runge Kutta stage.
Also in this test the dramatic improvement obtained with the projection along characteristic variables is striking.

\begin{figure}
\centering
\includegraphics[width=0.3\linewidth]{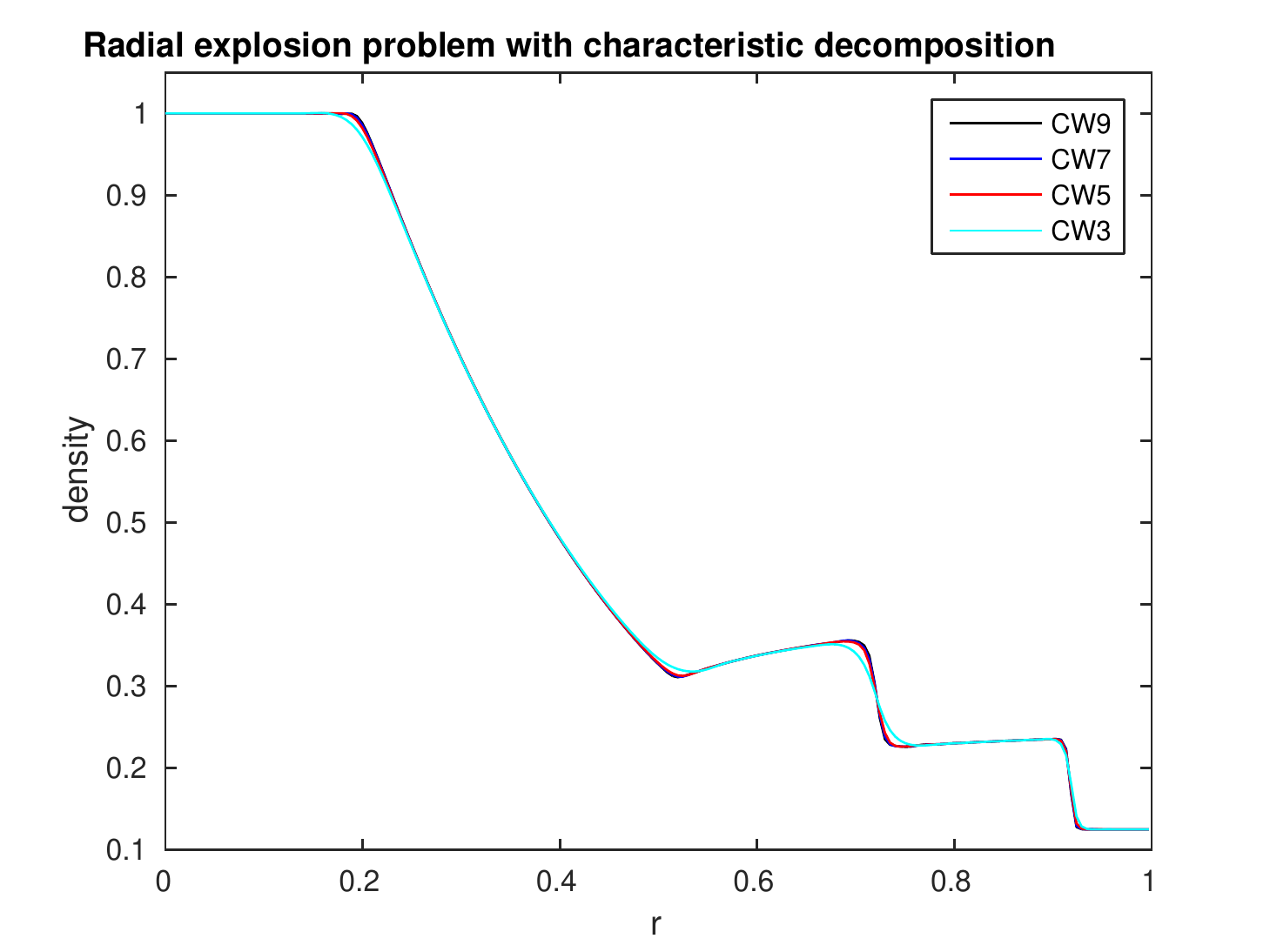}
\hfill
\includegraphics[width=0.3\linewidth]{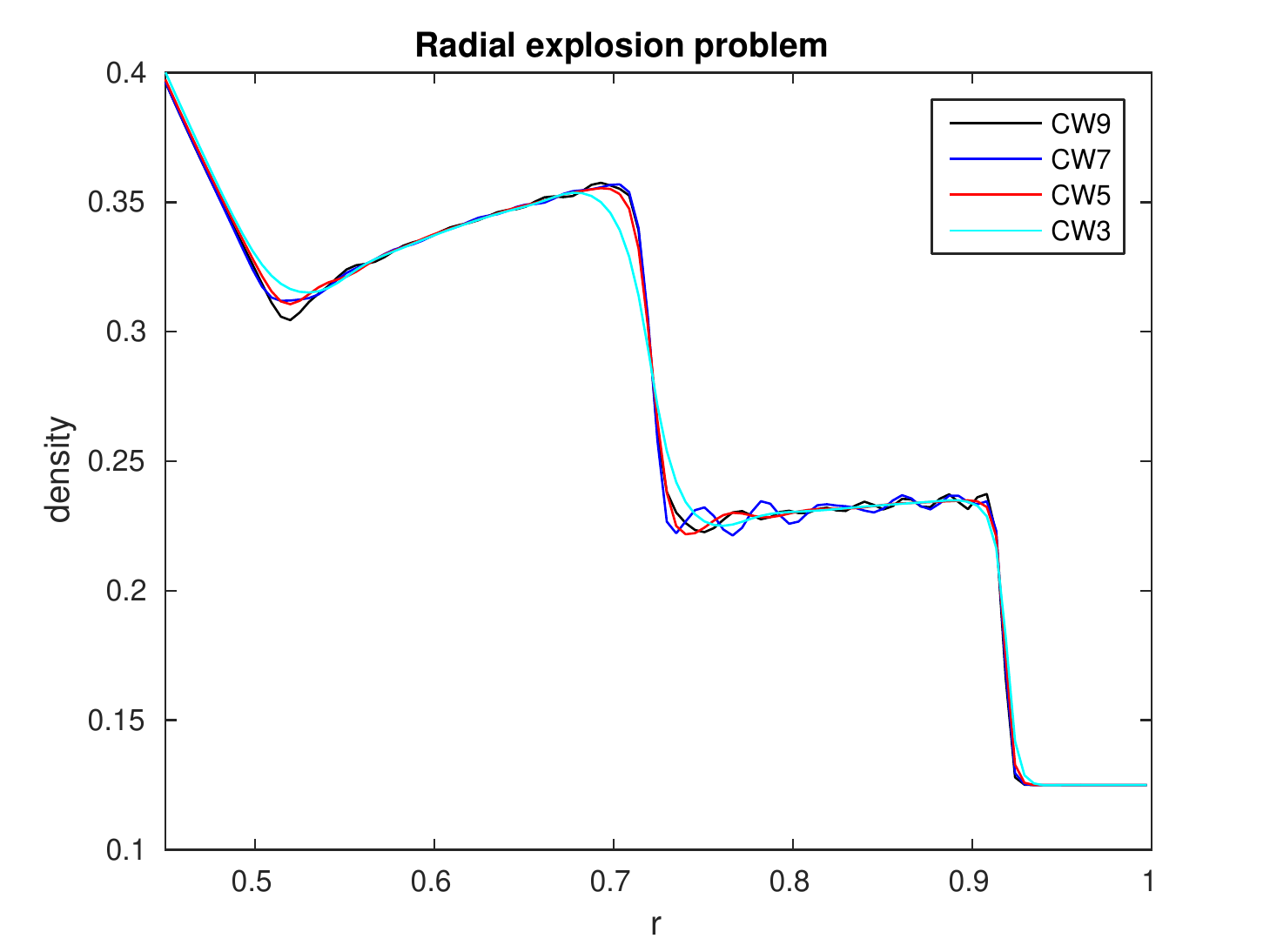}
\hfill
\includegraphics[width=0.3\linewidth]{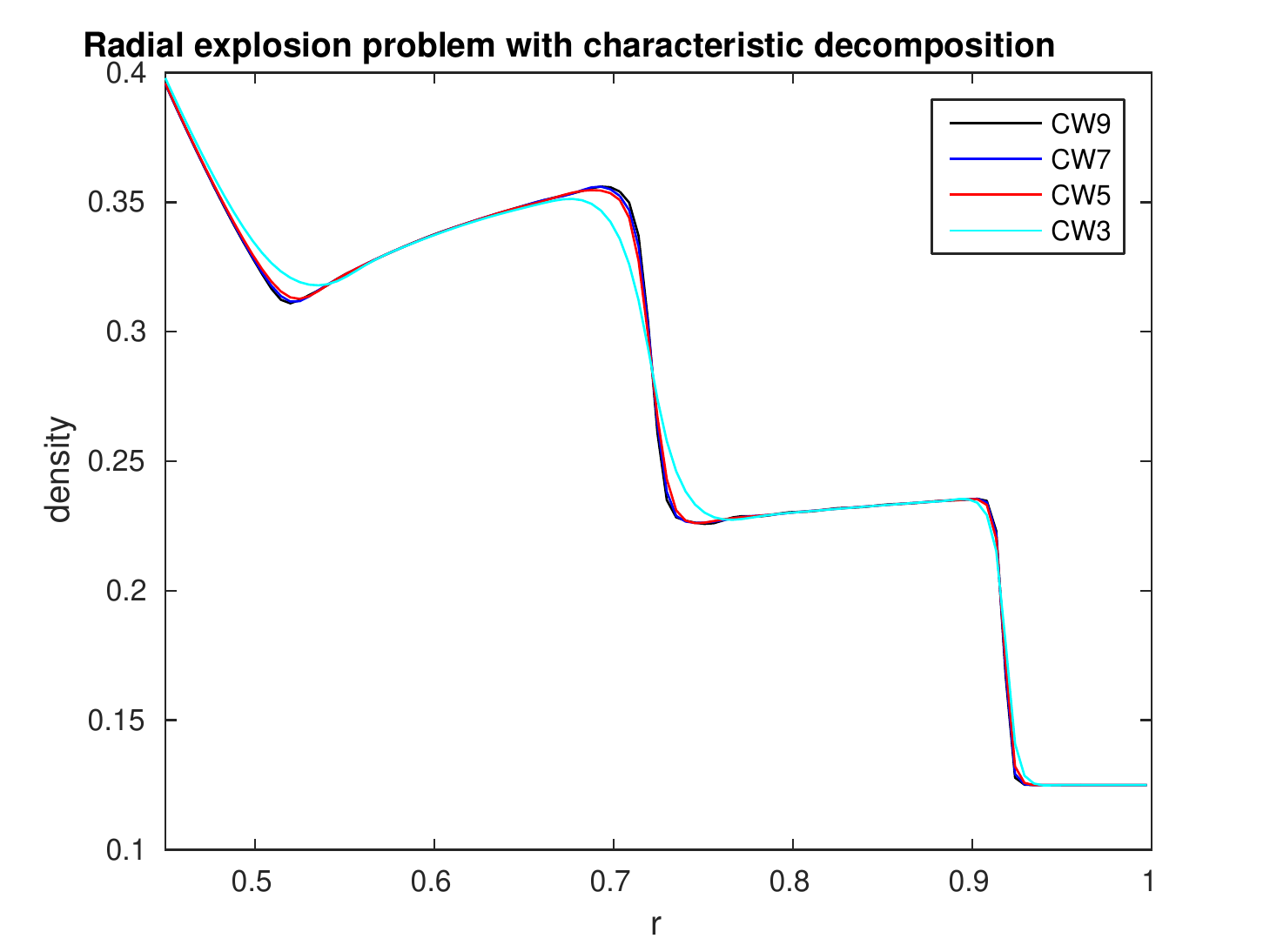}
\caption{Sod's explosion problem: density profiles for several $\CWENO$ schemes (left), zoom on the contact and shock wave with the reconstruction computed along conservative variables (middle), and along characteristic variables (right).}
\label{fig:radialSod}
\end{figure}

\section{Conclusions}   \label{sec:conclusions}
In this paper we introduced a class of spatial reconstruction procedures that are characterised by computing a reconstruction function whose accuracy is uniform across the whole cell, instead of reconstructed point values, as in the standard $\WENO$ reconstruction. This class of algorithms contains the already proposed $\CWENO3$ of \cite{LPR:99}, $\CWENO5$ of \cite{Capdeville:08} and the two-dimensional third order reconstruction of \cite{SCR:16}. 

In particular, within this framework, we focused on one-dimensional reconstruction procedures of any odd order $2g+1$ (which were never considered before for $g>2$) and proved that the nonlinear mechanism for stencil selection guarantees the desired accuracy of order $2g+1$ when the procedure is applied to smooth enough data. The non-oscillatory properties of the reconstruction in the presence of discontinuities in the input data are studied more deeply than in previous papers and a sufficient condition (property R) is given, to direct the choice of the parameters appearing in the reconstruction, to avoid spurious oscillations.
Moreover, it is shown that any the one-dimensional $\CWENO$ scheme satisfy property R.

We think that this is the first time that the potential of these reconstructions is explored in the case of balance laws, and their properties are systematically studied.

The new schemes perform on par with $\WENO$ reconstructions regarding accuracy on smooth data and the production of spurious oscillations close to discontinuities, but they are, in our opinion, more versatile than $\WENO$, because they result in a whole reconstructing polynomial which can be evaluated where needed. This is very important on balance laws, non uniform grids, moving mesh algorithms. In fact, in $\CWENO$ schemes,  the accuracy requirements involve only the degree of the candidate polynomials and not the values of suitably chosen linear coefficients. This means that, in a $\CWENO$ procedure, the linear coefficients can be chosen independently of the point at which the reconstruction is to be evaluated and independently of the relative size of the neighbouring cells. 

With these new schemes, unlike $\WENO$, it is possible to compute boundary value reconstructions on uniform or non-uniform grids (to compute numerical fluxes), and, at the same time, evaluate  the reconstruction at points in the interior of the computational cells, for evaluating quadratures of source terms, {\em with the same reconstruction polynomial}. The same polynomial can also be used to compute quantities that employ quadrature formulas in the cell, as in the initialisation of cell averages after a grid refinement on h-adaptive schemes or after mesh movement in moving mesh techniques. Another important application is the computation of cell averages of functions of the conserved variables arising in the computation of local residuals for a posteriori error control, as in the case of the numerical entropy error indicator. A very important application can be found in finite volume schemes for balance laws, in the computation of cell averages of source terms. This latter application in particular is tested in this paper, for accuracy orders up to 9.

In this paper we also introduce formulas to compute the reconstructions, in one space dimension, from the divided differences of the data in the case of non-uniform grids, and we provide tables of coefficients, obtained from undivided differences in the case of uniform grids. We note that the structure of these tables, whose entries do not depend on the degree of the polynomial to be computed, allows easily to raise or lower the degree of the reconstruction. The exploitation of this feature for p-adaptivity will be the subject for future work.

This paper is mainly concerned on $\CWENO$ reconstructions in one space dimension. The extension to multidimensional in the case of Cartesian grids is straightforward, but it is also possible to extend these techniques to unstructured grids.

\bibliographystyle{plain}
\bibliography{CWENOH}

\end{document}